\documentclass[10pt,psamsfonts]{amsart}
\usepackage[utf8]{inputenc}

\usepackage{amsmath}
\usepackage{amsthm}
\usepackage{amssymb}
\usepackage{amscd}
\usepackage{amsfonts}
\usepackage{amsbsy}
\usepackage{graphicx}
\usepackage[dvips]{psfrag}
\usepackage{array}
\usepackage{color}
\usepackage{xcolor}
\usepackage{epsfig}
\usepackage{url}
\usepackage{hyperref}
\usepackage{float}
\usepackage{subcaption}
\usepackage{stmaryrd}	
\usepackage{overpic}
\usepackage{epstopdf}
\usepackage{tikz, pgfplots}
\usepackage{epsfig,afterpage}
\usepackage[all]{xy}
\usepackage[makeroom]{cancel}
\usepackage{soul}
\usepackage{dsfont}

\makeatletter
\DeclareFontFamily{U}{tipa}{}
\DeclareFontShape{U}{tipa}{m}{n}{<->tipa10}{}
\newcommand{\ark@char}{{\usefont{U}{tipa}{m}{n}\symbol{62}}}%

\newcommand{\ark}[1]{\mathpalette\ark@arc{$#1$}}

\newcommand{\ark@arc}[2]{%
	\sbox0{$\m@th#1#2$}%
	\vbox{
		\hbox{\resizebox{\wd0}{\height}{\ark@char}}
		\nointerlineskip
		\box0
	}%
}
\makeatother

\newcommand{\R}{\ensuremath{\mathbb{R}}}

\newcommand{\N}{\ensuremath{\mathbb{N}}}

\newcommand{\Z}{\ensuremath{\mathbb{Z}}}

\newtheorem {theorem} {Theorem} 

\newtheorem {proposition} {Proposition}

\newtheorem {definition} {Definition}
\newtheorem {remark} {Remark}
\newtheorem {example} {Example}

\newtheorem {corollary} {Corollary}

\begin{document}
	\renewcommand{\arraystretch}{1.5}
	
	\title[Entropy and Gurevich Pressure for piecewise smooth vector fields]
	{Entropy and Gurevich Pressure for piecewise smooth vector fields} 
\author[M. Florentino ]
{Marco  Florentino  }

\address{Instituto de Ciências Puras e Aplicadas-UNIFEI, CEP  35903-087, Rua Irm\~{a} Ivone Drumond, 200, Distrito Industrial II, Itabira, Minas Gerais, Brazil} \email{marcoflorentino@unifei.edu.br}





\maketitle

\begin{abstract}
In this paper we provide a bridge between classical results concerning discrete dynamical systems and dynamical systems governed by nonsmooth vector fields. In fact, we obtain a set of piecewise smooth vector field trajectories where the time-one map is well defined. Afterwards, we obtain a topological conjugacy between the itinerary of a trajectory contained in this set and the one-sided shift space with a set of states countable with the discrete topology. We then use this fact to calculate the Gurevich pressure and entropy for non-smooth vector fields, and we also use the Lyapunnov function to estimate the mixing time.  Furthermore, we relate the topological entropy at infinity of the global trajectories of these non-smooth vector fields to the Hausdorff dimension of the set of recurrent points that escape on average. Finally, we approach it from the point of view of cofinite topology, thus obtaining relevant examples.
	\end{abstract}

\vspace{1cm}


\section{Introduction}
\label{intro}

The main goal of this paper is to formalize and obtain a way to calculate entropy and Gurevich pressure for piecewise smooth fields (PSVFs for short). PSVFs considered here does not present uniqueness of trajectories. Therefore, it is possible that a lot of trajectories are passing through a specific point of the phase space (see Definition \ref{def2} for the way that the trajectories are obtained). Assuming  that a particle passing through such a point has certain probabilities in following each one of the trajectories, this generates a kind of stochastic process when we consider what is the trajectory of a particle departing from such points.. In this way, it is possible to use symbolic dynamics in infinite alphabets to find out what happens to the trajectories of these fields, as long as there is an application between these two “worlds” that preserves certain properties.

Symbolic dynamical systems in finite alphabets are mathematical objects that provide a wealth of examples and have greatly influenced theoretical development in dynamical systems. In computer science, for example, topological Markov shift generated by finite graphs model the evolution of transition systems. The most important numerical invariant of dynamical systems is topological entropy. For these systems, entropy is equal to the exponential growth rate of the number of finite words of fixed length. In the case of a topological Markov shift, the entropy is equal to the natural logarithm of the spectral radius of the generating graph. Considering the graph as a linear map, the spectral radius measures the rate of dilation under the iterated application. On an exponential scale, this rate is equal to the growth in the number of finite words. In addition to theoretical interest, there is great interest in the field of computing due to the growing importance of infinite state systems.

A topological Markov shift  with a countable infinite state space is more complicated than a topological Markov shift with a finite state space, and inevitably, any attempt to study symbolic dynamics in infinite alphabets has to deal with the fact that the discrete topology employed in finite case, leads us to shift spaces that are not compact. Furthermore, such a topological Markov shift does not necessarily have a stationary distribution. Even if a stationary distribution exists and the topological Markov shift is mixing, there is no general algorithm for estimating mixing rates, however, it is possible to use a Lyapunov function to establish the existence of a steady-state distribution and estimate mixing rates. To obtain the formal results, we construct a topological Markov shift with a countable infinite state space formed by the trajectory arcs of a global trajectory of a PSVF. The topological Markov shift to be considered will be obtained using the time-one map. A huge objective is to prove the conjugacy between the time-one map of a PSVF and a one-sided shift and thus obtain the topological Markov shift. However, in a first view, a time-one map is not well defined in this context because, if we depart from a point where more than one trajectory is passing through, where will we stay after $t = 1$ ? We can stay either at a point of one trajectory or at a point of
the another one and in this way, the same trajectory could be mapped in more than one symbol. In order to address this issue, which is the most significant obstacle to overcome, we examine a quotient space and define a metric on it when using discrete topology. This approach produces a metric space that is not compact that exhibits a well-defined time-one map and topological Markov shifts under certain conditions for the PSVFs being considered.

The paper is organized as follows: in Section \ref{secao teoria basica CVSPs} we present the basic notions about PSVFs. In Section \ref{two-sided topological Markov shifts} we present the basic notions about two-sided topological Markov shifts. Next, in Section \ref{conjugacy}, we construct a PSVF and show that the quotient space of all global trajectories is topologically conjugate by the time-one map induced by a (sub)shift. In Section \ref{secao resultados principais} we present the main results of the article, which are: Theorem \ref{teo1} and its Corollary \ref{cor1} , where we establish conditions for obtaining the Gurevich pressure and entropy for the aforementioned PSVF. The Theorem \ref{teo2} and its Corollary \ref{cor2} show the conditions to obtain weak mixing, strong mixing and ergodicity for the aforementioned PSVF. Furthermore, the Theorem \ref{teoo3} establishes an upper limit for the mixing time, using a Lyapunov function. Finally, we present the Theorem \ref{teo-princ} and the Corollary \ref{cor-teo5}, showing that the canonical forms presented in the previous results represent a larger class of PSVFs conjugated to the mentioned (sub)shifts. In Section \ref{secao prova resultados principais} these results are proven. In  Section \ref{hausdorf-dim} we will relate the topological entropy at infinity of the global trajectories of $Z_{\infty}$ with the Hausdorff dimension of the set of recurrent points that escape on average. Finally in Section \ref{cofinity-topology}, we approach it from the point of view of cofinite topology, which as we know is not metrizable, but we still achieve a topological conjugacy, thus generating interesting examples for calculating Gurevich entropy. The Appendix \ref{apen} is dedicated to some technical (and already known) questions about stochastic processes and topological Markov shift used throughout the article.

\section{Preliminaries}\label{secao teoria basica CVSPs}
			
	Let $V$ be an open set of $\mathbb{R}^n$. Consider a manifold $\Sigma \subset V$ of codimension $1 $ in $\mathbb{R}^n$ given by $\Sigma = f^{-1} (0)= \left\lbrace q \in V : f(q) = 0 \right\rbrace $, where $f : V \rightarrow \mathbb{R} $ is smooth having $0 \in \mathbb{R}$ as a regular value (that is, $\nabla f(p) \neq 0$, for $p \in f^{-1} (0)$). We call $\Sigma$ a switching manifold whose boundary separates the regions $\Sigma^{+} = \left\lbrace q \in V : f(q) \geq 0 \right\rbrace $ and $\Sigma^{- } = \left\lbrace q \in V : f(q) \leq 0 \right\rbrace $.
				
Call $\mathfrak{X}^r$ the space of the $C^r$-vector fields in $V \subset \mathbb{R}^n$ endowed with the
$C^r$-topology, with $r \geq 1$ large enough for our purposes. Call $\mathcal{Z}^r$ the space of PSVFs $Z : V \rightarrow \mathbb{R}^n $ such that
				
	\begin{equation}
\label{sis}
	Z (q) = \left\{ \begin{array}{c}
	X^{+} (q) \, \, if \,\, q \, \in \Sigma^{+} \\
	X^{-} (q) \, \, if \,\, q \, \in \Sigma^{-}
	\end{array}, \right. 
	\end{equation}
where $X^{+} = \left( X_{1}^{+}, X_2^{+}, \cdots , X_n^{+} \right) $, $X^{-} = \left( X_1^{-}, X_2^{-}, \cdots , X_n^{-} \right) \in \mathfrak{X}^r .$ 
				
We denote \eqref{sis} simply by $Z = (X^{+}, X^{-})$ when there is no confusion about the switching manifold. Note that $Z$ is multi-valued in $\Sigma$. We equip $\mathcal{Z}^r$ with the product topology, i.e., $$\parallel Z \parallel_{C^r} = max \left\lbrace \mid X^{+} \mid_{C^r}, \mid X^{-} \mid_{C^r} \right\rbrace , $$ where $\mid \cdot \mid_{C^r}$ denotes the classical $C^r$-norm of the smooth vector fields $X^{+}$ and $X^{-}$ restricted to $\Sigma^{+}$ and $\Sigma^{-}$, respectively.
				
In order to define rigorously the flow of  $Z$ passing through a point $p \in V$, we distinguish whether this point is at $\Sigma^{\pm} \setminus \Sigma$ or $\Sigma$. For the first two regions, the local trajectory is defined by $X^{+}$ and $X^{-}$ respectively, as usual, but for $\Sigma$ we rely on the contact between the vector fields $X^{+},$ $X^{-}$ and $\Sigma$ characterized by the Lie derivative $X^{+} f(q) = \left\langle \nabla f(q), X^{+}(q) \right\rangle $, where $\left\langle \cdot, \cdot \right\rangle $ is the usual inner product. We also use higher order derivatives given by $(X^{+})^k f = X^{+} ((X^{+})^{k-1} f) = \left\langle \nabla (X^{+})^{k-1} f, X^{+} \right\rangle$, with $k > 1$ a positive integer.  
				
A \textbf{Crossing Region } is defined by $\Sigma^{c} = \left\lbrace p \in \Sigma \mid X^{+}f(q) X^{-}f(q) > 0 \right\rbrace $. In addition we denote  $$\Sigma^{c^+} = \left\lbrace p \in \Sigma \mid X^{+}f(q) > 0,  X^{-}f(q) > 0 \right\rbrace \; \textrm{and}\;\Sigma^{c^-} = \left\lbrace p \in \Sigma \mid X^{+}f(q) < 0,  X^{-}f(q) < 0 \right\rbrace .$$
This region $\Sigma^c$ is relatively open in $\Sigma$ and their definition exclude the points where \linebreak $X^{+}f(q) X^{-}f(q) = 0$. These points are on the boundary of this region.
				
	\begin{remark}
	In this paper we will not deal with points where $X^{+}f(q) X^{-}f(q) < 0$ and a sliding motion will not be considered.
	\end{remark} 
				
Any $q \in \Sigma$ such that $X^{+}f(q)X^{-}f(q) = 0$ is called a boundary singularity. The boundary singularities can be of two types: $(i)$ an equilibrium of $X^{+}$ or $X^{-}$ over $\Sigma$ or $(ii)$ a point where a trajectory of $X^{+}$ or $X^{-}$ is tangent to $\Sigma$ (and it is not an equilibrium of  $X^{+}$ or $X^{-}$). In the second case, we call $q \in \Sigma$ a tangential singularity (or tangency point) and we denote the set of these points by $\Sigma^t$. If there exists an orbit of the vector field $X^{+}\mid_{\Sigma^{+}}$ (respectively $X^{-}\mid_{\Sigma^{-}}$) reaching $q \in \Sigma^t$ in a finite time such that the trajectory continues in $\Sigma^{+}$ (respectively $\Sigma^{-})$, then such tangency is called a visible tangency for $X^{+}$ (respectively $X^{-}$), otherwise we call $q$ an invisible tangency for $X^{+}$ (respectively $X^{-}$). 
				
We may also distinguish a particular tangential singularity called two-fold, which is a 
	tangency $q$ of $X^{+}$ and $X^{-}$ simultaneously (that is, $X^{+}f(q) = X^{-}f(q) = 0$) satisfying $(X^{+})^ 2f(q) (X^{-})^2f(q) \neq 0$. A two-fold is called
\begin{enumerate}
	\item[$1.$] visible-visible, if it is a visible tangency for both $X^{+}$ and $X^{-}$;
	\item[$2.$] invisible-invisible, if it is an invisible tangency for $X^{+}$ and $X^{-}$;
	\item[$3.$] visible-invisible, whether it is a visible tangency for $X^{+}$ and an invisible tangency for $X^{-}$ or vice versa.
\end{enumerate} 
				
				
				
				%

				
\textbf{In the sequel of the paper, we will consider just planar PSVFs.} In this case, we say that a tangency point $p \in V$ is singular if $p$ is an invisible-invisible tangency for both $X^{+}$ and $X^{-}$. On the other hand, a tangency point $p \in V$ is regular if it is not singular.

\begin{definition}
\label{def2}
	The local trajectory (orbit) $\phi_Z (t, p)$ of a PSVF $Z = (X^{+}, X^{-})$ through a small neighborhood of $p \in U$ is defined as follows: \begin{enumerate}
	\item[$(i)$] For $p \in \Sigma^{+}$ and $p \in \Sigma^{-}$ the trajectory is given by $\phi_Z (t, p) = \phi_{X^{+}} ( t,p)$ and $\phi_Z (t,p) = \phi_{X^{-} }(t,p)$ respectively; 
	\item[$(ii)$] For $p \in \Sigma^{c^{+}}$ and taking the origin of time at $p$ the trajectory is defined as $\phi_Z (t, p) = \phi_{X^{-}} (t,p)$ for $t \leq 0$ and $\phi_Z (t,p) = \phi_{X^{+} }(t,p)$ for $t \geq 0.$ If $p \in \Sigma^{c^{-}}$ the definition is the same, but inverting time;
\item[$(iii)$] For $p$ a point of regular tangency and taking the origin of time at $p$ the trajectory is defined as $\phi_Z (t, p) = \phi_1 (t,p)$ for $t  \leq 0$ and $\phi_Z (t,p) = \phi_2 (t,p)$ for $t  \geq 0,$ where each $\phi_1, \phi_2$ is either $\phi_{X^{+}}$ or $\phi_{X^{-}}$; 
	\item[$(iv)$] For $p \in V \subset \mathbb{R}^2$ a singular tangency point, $\phi_Z (t, p) = p$ for all $t \in \mathbb {R}.$ 
\end{enumerate}
\end{definition} 
				

\begin{definition}
	\label{def1}
 A global trajectory $\Gamma_Z (t, p_0)$ is a concatenation of local trajectories. Moreover, a maximal trajectory  is a global trajectory that can not be extended to any other global trajectories by joining local ones, that is, if $\widetilde{\Gamma_Z} $ is a global trajectory containing $\Gamma_z $ then $\Gamma_z = \widetilde{\Gamma_Z}$. In this case, we call $I = (\tau^{-}(p_0), \tau^{+} (p_0))$ the maximal interval of the solution $\Gamma_Z$. A global trajectory is a \textbf{positive} (respectively, \textbf{negative}) global trajectory if $t > 0$ (respectively, $t < 0)$ and $t_0 = 0$.
\end{definition}

\begin{remark}
We should note that the maximal interval of the solution may not cover the interval $(- \infty , \infty)$, that is, $\tau^{\pm} (p_0)$ could be finite values.
\end{remark}

\begin{definition} 
An $\infty$-homoclinic loop of a planar PSVF is a global trajectory of $Z$ presenting $\infty$ visible-visible two-fold singularities $p_1, p_2, \cdots$ in such a way that, after passes through $p_i$, the trajectory reaches $\Sigma$ either in $p_{i-1}$ or $p_{i+1}$.
\end{definition}

\begin{definition}
Two PSVFs $Z, \widetilde{Z} \in \mathcal{Z}^r$, defined in $U, \widetilde{U}$ respectively and with switching manifold $\Sigma$ and $\widetilde{\Sigma}$ are $\Sigma$-equivalent if there exists an orientation preserving homeomorphism $h : U \to \widetilde{U}$ that sends $U \cap \Sigma$ to $\widetilde{U} \cap  \widetilde{\Sigma}$ , the orbits of $X^{+}$ restricted to $U \cap \Sigma^{+}$ to the orbits of $\widetilde{X}^{+}$ restricted to $\widetilde{U} \cap \widetilde{\Sigma}^{+}$, the orbits of $X^{-}$ restricted to $U \cap \Sigma^{-}$ to the orbits of $\widetilde{X}^{-}$ restricted to $\widetilde{U} \cap \widetilde{\Sigma^{-}}$. 
 
\end{definition}

In this paper, the smooth vector fields $X^{\pm}$ satisfy  $div(X^{\pm}) = 0 $ in $\Sigma^{\pm}$. In other words, they preserve the volume measure in $\Sigma^{\pm}$ (see \cite{viana}). Therefore, by the Corollary A of \cite{nova}, the PSVF $Z=(X^{+},X^-)$ preserves the volume measure, that is, the Lebesgue measure, here denoted by $med$.

\subsection{The set of all trajectories of PSVFs and the time-one map}


Let $Z = \left(X^{+}, X^{-} \right) $ be a PSVF defined over a compact $2$-dimensional surface $M$ and $\Omega = \{ \textrm{positive global}\\ \mbox{trajectory of} \, Z\}$.  Consider the function
$\rho : \Omega \times \Omega \to \R$ given by
\begin{equation*}
	\rho(\gamma_1,\gamma_2)=\underset{i \in \N }{\sum} \frac{1}{2^i }\int_{i}^{i+1} \mid\gamma_1(t)-\gamma_2(t)\mid\,dt,
\end{equation*}	
where $\mid\gamma_1(t)-\gamma_2(t)\mid$ denotes the distance between the points $\gamma_1(t)$ and $\gamma_2(t)$.

\begin{proposition}
	The space $(\Omega, \rho)$ is a metric space.
\end{proposition}

\begin{proof}
Let $\gamma_1, \gamma_2 \in \Omega$. Observe that M being compact, implies $|\gamma(t)-\gamma_2(t)|$ is bounded for all $\in \R$, thus the series above converges for any $\gamma_1, \gamma_2$.	If $\rho(\gamma_1, \gamma_2) = 0$ then $\int_{i}^{i+1} |\gamma_1(t) - \gamma_2(t)|\, dt = 0$ for all $i \in \N$ which implies $\gamma_1(t) = \gamma_2(t)$ for all $t \in \R$ and therefore $\gamma_1 = \gamma_2.$	The fact $\rho(\gamma_1, \gamma_2)= \rho(\gamma_2, \gamma_1)$ follows immediately from $|\gamma_1(t) - \gamma_2(t)| =
|\gamma_2(t) - \gamma_1(t)|$. And, finally, for the triangle inequality part it is enough to notice that
	$|\gamma_1(t) - \gamma_3(t)| \leq |\gamma_1(t) - \gamma_2(t)|+|\gamma_2(t) - \gamma_3(t)|$ for all $t \in \R$ gives the inequality $\rho(\gamma_1, \gamma_3) \leq \rho(\gamma_1, \gamma_3) + \rho(\gamma_2, \gamma_3).$
\end{proof}

 Consider the map: 
\begin{equation}
	\label{acao-grupo}	
	\left. 
	\begin{array}{cc}
		T : \mathbb{R} \times \Omega \to  \Omega\\
		\quad \quad \quad \quad \quad \quad \quad \quad \quad (t,\gamma)\mapsto T(t,\gamma)(\cdot) = \gamma(\cdot + t) .
	\end{array}
	\right. 
\end{equation}	

Then we have the time-one  map  $T_1(\gamma) = \gamma( \cdot + 1).$  

\begin{proposition}
	The map $T_1 : \Omega \to \Omega$ defined above is continuous.
\end{proposition}
\begin{proof}
	Note that $$\int_{i}^{i+1} \mid T_1(\gamma_1)(t)- T_1(\gamma_2)(t)\mid\,dt = \int_{i}^{i+1} \mid\gamma_1(t+1)-\gamma_2(t+1)\mid\,dt =$$ $$ \int_{i+1}^{i+2} \mid\gamma_1(t)-\gamma_2(t)\mid\,dt.$$
	Using the relation above, we obtain:
	\begin{equation*}
		\rho\left( T_1(\gamma_1)(t), T_1(\gamma_2)(t)\right) = \underset{i \in \mathbb{N}}{\sum} \frac{1}{2^{ i } }\int_{i}^{i+1} \mid T_1(\gamma_1)(t)- T_1(\gamma_2)(t)\mid\,dt =  
	\end{equation*}
	\begin{equation*}
		\underset{i \in \mathbb{N}}{\sum} \frac{1}{2^{ i } }\int_{i+1}^{i+2} \mid \gamma_1(t)- \gamma_2(t)\mid\,dt = \underset{n \to +\infty}{\lim} \left( \overset{n}{\underset{i = 0 }{\sum}} \frac{1}{2^{ i } }\int_{i+1}^{i+2} \mid \gamma_1(t)- \gamma_2(t)\mid\,dt\right) = 
	\end{equation*}
	\begin{equation*}
		\underset{n \to +\infty}{\lim} \left( 2\overset{n}{\underset{i = 0 }{\sum}} \frac{1}{2^{ i +1 } }\int_{i+1}^{i+2} \mid \gamma_1(t)- \gamma_2(t)\mid\,dt\right) =  \underset{n \to +\infty}{\lim} \left(2 \overset{n+1}{\underset{j = 1 }{\sum}} \frac{1}{2^{j} }\int_{j}^{j+1} \mid \gamma_1(t)- \gamma_2(t)\mid\,dt\right)
	\end{equation*}
	\begin{equation*}
		\leq \underset{n \to +\infty}{\lim} \left( 2\overset{n+1}{\underset{j = 1 }{\sum}} \frac{1}{2^{j} }\int_{j}^{j+1} \mid \gamma_1(t)- \gamma_2(t)\mid\,dt\right) + \int_{0}^{1} \mid \gamma_1(t)- \gamma_2(t)\mid\,dt  = 
	\end{equation*}
	\begin{equation*}
		2 \underset{n \to +\infty}{\lim} \left( \overset{n+1}{\underset{j = 1 }{\sum}} \frac{1}{2^{j} }\int_{j}^{j+1} \mid \gamma_1(t)- \gamma_2(t)\mid\,dt + \frac{1}{2} \int_{0}^{1} \mid \gamma_1(t)- \gamma_2(t)\mid\,dt \right)  = 2 \rho(\gamma_1,\gamma_2) .
	\end{equation*}
	Hence $T_1$ is continuous.
\end{proof}

\section{One-sided topological Markov shifts, Entropy Gurevich and Gurevich Pressure }
\label{two-sided topological Markov shifts}
Suppose $\mathcal{G} = \mathcal{G}(V,E)$ is a directed graph with a  countable collection of vertices $V$ and edges $E \subset V \times V$. 
The notation $a \to b$ means that $(a, b) \in E$, and the notation $a \stackrel{n}{\to} b$ means that there are vertices $i_1, \cdots i_{n-1}$, such that $a\to i_1 \to\cdots i_{n-1}\to b$. When $a$ connects to $b$ in $n$ steps, we say that it is a \textbf{path} of length $n$ and that $(a,i_1, \cdots, i_{n-1}, b) $ is \textit{admissible}. This path is called \textbf{loop} if $a = b$.

%
%
\begin{definition}
	\label{defshift}
The \textbf{one-sided topological Markov shift}  associated to $\mathcal{G}$ is the set $\mathcal{A}_{\mathcal{G}}^{\mathbb{N}} := \left\lbrace (a_i)_{i \in \mathbb{N}} \in V^{\mathbb{N}}\,; \; a_i \to a_{i+1}\;\; \mbox{for all}\; i\right\rbrace$. Given  $a = (a_j)_{j \in \mathbb{N}}$ and $b = (b_j)_{j \in \mathbb{N}}$ in $\mathcal{A}_{\mathcal{G}}^{ \mathbb{N}}$. Define $d : \mathcal{A}_{\mathcal{G}}^{\mathbb{N}} \times \mathcal{A}_{\mathcal{G}}^{\mathbb{N}} \rightarrow \mathbb{R}$ by: 
\begin{equation}
\label{metrica}
d(a, b) = \underset {i \in \mathbb{N}}{\sum} \frac{\mid a_i - b_i \mid}{2^{i}}.
\end{equation}

$d$ is a metric, which makes $\mathcal{A}_{\mathcal{G}}^{\mathbb{N}}$ a metric space. Additionally we can set the map shift  $ \sigma_{\mathcal{G}} : \mathcal{A}_{\mathcal{G}}^{\mathbb{N}} \rightarrow \mathcal{A}_{\mathcal{G}}^{\mathbb{N}} $ given by $\sigma_{\mathcal{G}} ((a_j)) = b_j$, where $b_j = a_{j+1}.$ The map  is called two-sided full shift and the discrete flow $(\mathcal{A}_{\mathcal{G}}^{\mathbb{N}}, \sigma_{\mathcal{G}})$ is called symbolic flow or shift system.
\end{definition}


$V$ is called the alphabet, and elements of $V$ are called states or letters. The set $V$ is endowed with the discrete topology, and $\mathcal{A}_\mathcal{G}^{\mathbb{N}}$ is endowed with the induced topology of $V^{\mathbb{N}}$.  The matrix $(t_{ab})_{V \times V}$, $t_{ab} = 1$ when $a \to b$ and $t_{ab} = 0$ when $a \nrightarrow b$, is called the transition matrix.
$(\mathcal{A}_{\mathcal{G}}^{\mathbb{N}}, d)$ is a complete and separable metric space. Its topology is generated  with the topology generated by the  cylinders
by the cylinders $$[a_0a_1 \cdots  a_{n-1}]_m = \left\lbrace (x_j)_{j \in \mathbb{N}} \in \mathcal{A}_{\mathcal{G}}^{\mathbb{N}}; \; \;  (x_m, \cdots, x_{m+n-1}) = (a_0, \cdots,a_{n-1})\;\; m , n \in \mathbb{N} 
\right\rbrace.$$

The cylinders are open and closed. The disjoint unions of the cylinders form an algebra that generates the $\sigma_{\mathcal{G}}$-Borel algebra  $\mathcal{B}$ in $\mathcal{A}_{\mathcal{G}}^{\mathbb{N}}.$ Since $V$ is infinite $\mathcal{A}_{\mathcal{G}}^{\mathbb{N}}$ is not compact, and sometimes not even locally compact. In fact, it can be shown that $\mathcal{A}_{\mathcal{G}}^{\mathbb{N}}$ is locally compact iff for every $a \in V , \, \# \left\lbrace b \in V\,  ; a \to b \right\rbrace  < \infty.$

\begin{definition}
A topological Markov shift $\mathcal{A}_{\mathcal{G}}^{\mathbb{N}}$ is called topologically transitive,
if for all $a, b \in V$, there is $n$ such that $a \stackrel{n}{\to} b$, and topologically mixing if for all $a, b \in V$,
there is $n$ such that $N = N(a, b)$ such that $a \stackrel{n}{\to} b$ for all $n \geq N.$
\end{definition}
 $\mathcal{A}_{\mathcal{G}}^{\mathbb{N}}$ is topologically transitive iff $\sigma_{\mathcal{G}}$ is topologically transitive (i.e. for all open sets $U, V$ there exists $n > 0$ such that $U \cap\sigma_{\mathcal{G}}^{-n}(V ) \neq \emptyset)$. Similarly, $\mathcal{A}_{\mathcal{G}}^{\mathbb{N}}$ is topologically mixing iff $\sigma_{\mathcal{G}}$ is topologically mixing (i.e. for all open sets $U, V$ there exists $N = N(U, V ) > 0$ such that $U \cap \sigma_{\mathcal{G}}^{-n}(V ) \neq \emptyset$ for all $n > N$).
\begin{definition}
Let $\mathcal{G}$ be an oriented graph, and let $a, b$ be two vertices in $\mathcal{G}$. We
define the following quantities:
\begin{enumerate}
	\item[$\bullet$] $p_{ab}^{\mathcal{G}}(n)$ is the number of paths $(u_0, u_1, \cdots u_n)$ such that $u_0 = a$ and $u_n = b$;
	\item[$\bullet$] $R_{ab}(\mathcal{G})$ is the radius of convergence of the series $\sum p_{ab}^{\mathcal{G}}(n)z^n$;
	\item[$\bullet$] $f_{ab}(\mathcal{G})$ is the number of paths $(u_0, u_1, \cdots u_n)$ such that $u_0 = a$, $u_n = b$, and
	$u_i \neq b$ for all $0 < i < n, L_{ab}(\mathcal{G})$ is the radius of convergence of the series $\sum f_{ab}(n)z^n.$
\end{enumerate}
\end{definition}

\begin{definition}
Let $\mathcal{G}$ be an oriented graph and $V$ its set of vertices. The graph $\mathcal{G}$
is \textbf{strongly connected} if, for all $a, b \in V$, there exists a path from $a$ to $b$ in $\mathcal{G}$.
The period of a strongly connected graph $\mathcal{G}$ is the greatest common divisor of $\left( p_{aa}^{\mathcal{G}}(n)\right)_{a \in V}$. The graph $\mathcal{G}$ is \textbf{aperiodic} if its period is $1$.
\end{definition}

\begin{proposition}[\cite{vere}]
Let $\mathcal{G}$ be an oriented graph. If $\mathcal{G}$ is strongly
connected, $R_{ab}(\mathcal{G})$ does not depend on $a$ and $b$, it is denoted by $R(\mathcal{G})$.
\end{proposition}

%
%



\begin{definition}
If $\mathcal{G}$ is a countable graph, the \textbf{Gurevich entropy}  of the topological Markov shift $\mathcal{A}_{\mathcal{G}}^{\mathbb{N}}$ (or graph $\mathcal{G}$)  is given by
\begin{equation*}
h_G(\mathcal{A}_{\mathcal{G}}^{\mathbb{N}}, \sigma_{\mathcal{G}}):= \left\lbrace h_{top}\left( \mathcal{A}_{\mathcal{H}}^{\mathbb{N}}, \sigma_{\mathcal{H}}\right); \sigma_{\mathcal{H}} = \sigma_{\mathcal{G}}|_{\mathcal{H}}, \mathcal{H} \subset \mathcal{G}\, \mbox{finite}  \right\rbrace .
\end{equation*}
\end{definition}

This entropy can also be computed in a combinatorial way, as the exponential
growth of the number of paths with fixed endpoints.
\begin{proposition}[\cite{gurev}]
Let $\mathcal{G}$ be a strongly connected oriented graph.
Then for all vertices $a, b \in V$ 
\begin{equation*}
	h_{{G}}(\mathcal{A}_{\mathcal{G}}^{\mathbb{N}},\sigma_{\mathcal{G}}) = \underset{n \to \infty}{\lim} \dfrac{1}{n} \log\left( p_{ab}^{\mathcal{G}}(n)\right) = -\log\left(R(\mathcal{G}) \right).  
\end{equation*}	
\end{proposition}

More than that, we still have the following result for topological Markov shifts.

\begin{theorem}[\cite{gurevi}]
Let $\mathcal{G}$ be an oriented graph. Then 
\begin{equation*}
h_G(\mathcal{A}_{\mathcal{G}}^{\mathbb{N}},\sigma_{\mathcal{G}}  ) = \sup \left\lbrace h_{\mu}\left( \mathcal{G} \right); \mu \, \mbox{is}\; \sigma_{\mathcal{G}}-\mbox{invariant probability measure}  \right\rbrace. 
\end{equation*}
\end{theorem}

The supremum is not necessarily reached. The next Theorem gives a necessary and sufficient condition for the existence of a measure of maximal entropy (that is, a probability measure $\mu$ such that $h_G(\mathcal{A}_{\mathcal{G}}^{\mathbb{N}})= h_{\mu}(\mathcal{A}_{\mathcal{G}}^{\mathbb{N}})$ when the graph is strongly connected.

\begin{theorem}[\cite{gurev}]
Let $\mathcal{G}$ be a strongly connected oriented graph of finite positive entropy. Then the topological Markov shift on $\mathcal{G}$ admits a measure of maximal entropy if and only if the graph $\mathcal{G}$ is positive recurrent. Moreover, such a measure is unique if it exists.
\end{theorem}

Let $\pi = \{p_a , a \in \mathcal{A}_{\mathcal{G}}\}$ be a probability vector assigning probabilities to the symbols of the alphabet $V$. For each block of size $k$ we have $$\mu_{\pi} ([a_0 a_1  \cdots a_{k-1}]) = p_{a_0} p_{a_{1}} \cdots p_{a_{k-1}},$$ where a block over $\mathcal{A}_{\mathcal{G}}^{\mathbb{N}}$ is a finite sequence of elements of $\mathcal{A}_{\mathcal{G}}^{\mathbb{N}}$. In this case we can see that a centered cylinder in the block of size $k,$ $a_0a_1 \cdots a_{k-1}$ at position $j$ is the set of all bi-infinite sequences of elements that match $a_0$ at position $j$, $a_1$ at position $j + 1, \cdots , a_n$ at position $j + k-1$. Note that if we treat the symbols appearing at the consecutive coordinates as $\mathcal{A}_{\mathcal{G}}^{\mathbb{N}}$-valued random variables then the product measure $\mu_{\pi}$ corresponds to the case where the above random variables are independent and identically distributed (i.i.d.). 

In ergodic theory, the measure $\mu_{\pi}$ is called the \textbf{Bernoulli measure} (associated to $\pi$), and the system $(\mathcal{A}_{\mathcal{G}}^{\mathbb{N}}, \mathcal{B}, \mu_{\pi} , \sigma_{\mathcal{G}})$ is called a Bernoulli shift.

\begin{definition}
Given $g : M \to M$ and any measure $\mu$ in $M$, we denote by $g_{\ast}\mu$ and call  iterated (or image) of $\mu$ by $g$, the measure defined by $g_{\ast}\mu(B)= \mu(g^{-1}(B))$ for each measurable set $B \subset M$. Note that $\mu$ is invariant to $g$ if and only if $g_{\ast}\mu = \mu$.
\end{definition}

\begin{definition}
	\label{transf}
Let $\mu$ and $\nu$ be probability measures invariant under measurable  maps $g : M \to  M$ and $h : N \to N$, respectively. We say that the systems $(M,\mathfrak{M},\mu,g)$ and $(N,\mathfrak{N}, \nu,h)$ are \textbf{ergodically equivalent} if		
\begin{enumerate}
\item[$(i)$] One can find measurable sets $A \subset M$ and $B \subset N$ that are invariant by $g$ and $h$ respectively, i.e., $g(A) \subset A$ and $h(B) \subset B$, with $\mu(A) = 1$ and $\nu(B) = 1$;
\item[$(ii)$]  There exists a bijection $L : M \to N$ such that $L$ and $L^{-1}$ are measurable, in such a way that, $h_{\ast}\mu = \nu$ and $L \circ g = h \circ L$.

	\end{enumerate}
\end{definition}
We say that a mensurable map $g : M \to M$ is a \textbf{Bernoulli map}  if $(M,\mathfrak{M},\mu,g)$ it ergodically equivalent to $\left(\mathcal{A}_{\mathcal{G}}^{\mathbb{N}},\mathcal{B},\mu_{\pi} , \sigma_{\mathcal{G}} \right)$.

A continuous function $\psi : \mathcal{A}_{\mathcal{G}}^{\mathbb{N}}\to {\mathbb{R}}$ is called a potential. For every $n \geq 1$, the $n$-variation of $\psi$ is defined by $$Var_n \psi:= \sup \left\lbrace \mid \psi(x)- \psi(y)\mid\, ; x,y \in \mathcal{A}_{\mathcal{G}}^{\mathbb{N}}, \, x_i=y_i, \, \mid i \mid \leq n-1, \; n \geq 1 \right\rbrace .$$ We say that $\psi$ has \textbf{summable variations} if $$\displaystyle \sum _{n =2}^{\infty }Var_n \; \psi < \infty.$$
Finally, we say that $\psi$ satisfies \textbf{Walters' condition} $\underset{n \in \mathbb{N}^{\ast}}{\sup}\,\underset{n+k}{Var}\;\psi_n < \infty$ for each $k \in \mathbb{N}$ and $\underset{k}{\lim}\,\underset{n \in \mathbb{N}^{\ast}}{\sup}\,\underset{n+k}{Var}\; \psi_n = 0.$

In $\left(\mathcal{A}_{\mathcal{G}}^{\mathbb{N}}, d\right)$ we denote by $\mathfrak{C}\left(\mathcal{A}_{\mathcal{G }}^{\mathbb{N}}\right)$ to the normed vector space of bounded functions $g : \mathcal{A}_{\mathcal{G}}^{\mathbb{N}} \to \R$, with the norm $\parallel g \parallel_{\infty} = \sup\left\lbrace \mid g (x)\mid : x \in \mathcal{A}_{\mathcal{G}}^{\mathbb{N}} \right\rbrace.$ Define the  operator $\mathcal{L}_{\psi}g: \mathfrak{C}\left(\mathcal{A}_{\mathcal{G }}^{\mathbb{N}}\right)\to \mathfrak{C}\left(\mathcal{A}_{\mathcal{G }}^{\mathbb{N}}\right) $, called the Ruelle's operator, as follows: 
\begin{equation}
\mathcal{L}_{\psi}g = \displaystyle \sum_{\sigma_{\mathcal{G}}(y)=x} e^{\psi}(y)g(y).
\end{equation}
Some calculations show that $$\mathcal{L}_{\psi}^ng = \displaystyle \sum_{\sigma_{\mathcal{G}}^n(y)=x} e^{\psi_n}(y)g(y),$$ where $\psi_n(x) = \overset{n-1}{\underset{j = 0}{\sum}} \psi(\sigma_{\mathcal{G}}^j(x))$ is the Birkhoff (ergodic) sum of $\psi.$ 

\begin{definition}
Given a potential $\psi : \mathcal{A}_{\mathcal{G}}^{\mathbb{N}} \to \mathbb{R}$ and $a \in V$, the $n$-\textit{th partition function} is defined by 
\begin{equation*}
\mathfrak{Z}_n\left( \psi, [a]\right) = \underset{x \in \mathcal{A}_{\mathcal{G}}^\mathbb{N}: \sigma_{\mathcal{G}}^n(x) =x}{\sum} e^{\psi_n(x)}{\Large \mathds{1}}_{[a]}(x),
\end{equation*}
where 
 ${\Large \mathds{1}}_{[a]}$ is the \textit{indicator function (or characteristic function)}. The \textbf{Gurevich pressure} at the symbol a is defined as 
\begin{equation*}
	P_G\left( \psi, [a]\right) := \underset{n \to \infty}{\limsup}\frac{1}{n} \log \mathfrak{Z}_n(\psi_n,[a]) 
	\end{equation*}
\end{definition}

When $\mathcal{A}_{\mathcal{G}}^{\mathbb{N}}$ is topologically   transitive and $\psi$ satisfies the Walters' condition, then $P_G(\psi, [a])$ does not depend on $a$. Since these hypotheses are the minimal assumed over the countable Markov shift and the potential, we will write $P_G(\psi)$. A particular but important case is to consider the constant potential $\psi \equiv 0$, then the pressure coincides with the Gurevich entropy, see \cite{gure, gurevi}:
$$h_G(\mathcal{A}_{\mathcal{G}}^{\mathbb{N}}) = P_G(0).$$ Moreover, if $\mathcal{A}_{\mathcal{G}}^{\mathbb{N}}$ is topologically mixing, then the pressure is a limit.

\begin{definition}
Suppose $\mathcal{A}_{\mathcal{G}}^{\mathbb{N}}$ is a countable Markov shift. A continuous function $\psi : \mathcal{A}_{\mathcal{G}}^{\mathbb{N}} \to \mathbb{R}$ is called \textit{admissible} if $\underset{\sigma_{\mathcal{G}}^n(y) =x}{\sum} e^{\psi_n(x){\Large \mathds{1}}_{[a]}(x)} < \infty$ for all $n \in \mathbb{N}$, $a \in V$, and $x \in \mathcal{A}_{\mathcal{G}}^{\N}$. 
\end{definition}

The following definitions will be important for what we will see in Section \ref{hausdorf-dim}.
\begin{definition}
Let $A \subset \mathcal{A}_{\mathcal{G}}^{\mathbb{N}}$. Denote by $R_A(x) :={\Large \mathds{1}}_A(x) \inf\{n \geq 1 : \sigma_{\mathcal{G}}^n(x) \in A\}$ the first return time map to the set $A$. In the particular case in which the set $A$ is a cylinder $[a]$ we denote $R_{[a]}(x).$
\end{definition}
Sarig \cite{sarig} introduced the following:

$$\mathfrak{Z}_n^{\ast}\left( \psi, [a]\right) = \underset{x \in \mathcal{A}_{\mathcal{G}}^\mathbb{Z}: \sigma_{\mathcal{G}}^n(x) =x}{\sum} e^{\psi_n(x)}{\Large \mathds{1}}_{[R_{[a]} = n]}(x),$$ where $[R_{[a]}=n] :=\{x \in \mathcal{A}_{\mathcal{G}}^{\mathbb{N}} :R_{[a]}(x)=n\}$.

\begin{definition}
Let $\left(\mathcal{A}_{\mathcal{G}}^{\mathbb{N}}, \sigma_{\mathcal{G}}\right) $ be topologically transitive countable Markov shift and $a \in V$. $$D_{\infty}([a]):= \underset{n \to \infty}{\limsup}\, \dfrac{1}{n}\mathfrak{Z}_n^{\ast}\left(0, [a]\right),$$ and $$D_{\infty}:= \underset{a \in \mathbb{N}}{\inf} D_{\infty}([a]).$$
\end{definition}
\begin{definition}
[Strongly positive recurrent countable Markov shift] Let $\left(\mathcal{A}_{\mathcal{G}}^{\mathbb{N}}, \sigma_{\mathcal{G}}\right) $ be a topologically transitive countable Markov shift with finite Gurevich entropy. We say that $\left(\mathcal{A}_{\mathcal{G}}^{\mathbb{N}}, \sigma_{\mathcal{G}}\right) $ is strongly positive recurrent if $D_{\infty} < h_{G}\left( \mathcal{A}_{\mathcal{G}}^{\mathbb{N}}\right)$. 
\end{definition}

We end this section by presenting the following definition

\begin{definition}
Given a topological Markov shift with the state space $V$, a function
$\mathcal{V}: V \to [1, \infty)$ is defined to be a \textbf{Lyapunov function} if there
exists a state $a^{\ast} \in V$ and parameter $0 < \lambda < 1$ (called drift) such that
$w(a^{\ast}|a^{\ast}) > 0$ and for any state $a\neq a^{\ast}$

\begin{equation}
	\mathbb{E}[\mathcal{V} (\mathcal{X}_{m+1})|\mathcal{X}_m] \leq\mathcal{V}(x)\quad m \in\N.
\end{equation}
\end{definition}

This definition of the Lyapunov function is more restrictive compared to that of \cite{meyn}, but this way the analysis becomes significantly simpler.

The existence of a Lyapunov function guarantees the existence of
a stationary distribution (see \cite{mt}). Furthermore, certain mixing properties are also maintained.
\section{Conjugacy between time-one map of  PSVFs  and  one-sided shift map}
\label{conjugacy}

Consider the following PSVF

\begin{equation}
	\label{zi}	
{Z_{\infty}} (x, y) =  \left\{ \begin{array}{cc} \left.\begin{array}{l} 
				X^{+}_{\infty} (x,y)   \; = \; \left( 1, 2 \sin (2 \pi x ) \right)  \; \textrm{for}  \; y \; \geq \; 0  \\
				X^{-}_{\infty} (x, y) \; = \; \left( -1, 2 \sin (2 \pi x ) \right)  \; \textrm{for}  \; y \; \leq \; 0 
			\end{array} \right. .
		\end{array}
		\right. 
\end{equation}

Take $P_{\infty} (x) = 1 - \cos (2 \pi x)$. Note that $$P_{\infty}^{'} (x) = 2 \pi \sin (2 \pi x) \Rightarrow P_{\infty}^{'} (j) = 0  \quad \forall \; j \in \mathbb{Z}$$ and $$P_{\infty}^{''} (x) = 4 \pi^2 \cos (2 \pi x) \Rightarrow P_{\infty}^{''} (j) = 4 \pi^2 > 0  \quad \forall \; j \in \mathbb{Z} .$$ 

The points $(p_j , 0)$ are visible-visible two-folds of $Z_{\infty}$, $j \in \mathbb{Z}.$ Let $$ \gamma_{\infty}^{X^{+}_{\infty}} = \{ (x, P_{\infty}(x)) \; | \: x \in \mathbb{R}\}\; \mbox{and} \;  \gamma_{\infty}^{X^{-}_{\infty}} = \{ (x, - P_{\infty}(x)) \; | \: x \in \mathbb{R}\} .$$  Define $\Lambda_{\infty} = \gamma_{\infty}^{{X^{+}_{\infty}}} \cup \gamma_{\infty}^{{X^{-}_{\infty}}} .$ Again note that $\Lambda_{\infty}$ is an invariant set for $Z_{\infty}.$ For more details see \cite{andre1}. 

\begin{remark}
We will show the existence of a conjugacy between the time-one map and a two-sided shift space. 
Since a PSVF may not provide uniqueness of trajectory through a point, the time-one map of  ${Z}_{\infty}$, ${Z}_{\infty}^1: \Lambda_{\infty} \to \Lambda_{\infty}$ such that  ${Z}_{\infty}^1(x) = \phi_{{Z}_{\infty}}(1, x)$, (here  $\phi_{{Z}_{\infty}}(0,x) = x$ and $\phi_{{Z}_{\infty}}$ is the flow of ${Z}_k$), is not well-defined, because it may have more than one image (depending on the flow chosen). One way of avoiding this is to work with a subset of the space of all possible trajectories in such a way that $T_1$ will be well-defined when restricted to this subset. So, consider $$\Omega_{\infty} = \left\lbrace \gamma \; | \; \gamma \; \mbox{ positive global trajectory of } {Z}_{\infty} \mbox{ with } \gamma (0) \in \Lambda_{\infty} \right\rbrace,$$ and then we can define the time-one map in $\Omega_{\infty}$, similarly to what was done before, that is, $T_1 : \Omega_{\infty} \to \Omega_{\infty}, \; T_1(\gamma)(\cdot) = \gamma(\cdot + 1).$
	
\end{remark}
	
	\begin{proposition}
		\label{pro}
Let $\gamma \in \Omega_{\infty}$, then for all $t \in \mathbb{R}$, there exists an unique $t^{\ast} \in [t, t+1)$ such that $\gamma (t^{\ast}) \in \gamma(t^{\ast}) \in \{(p_j,0)\}_{j \in \mathbb{N}}$.	
\end{proposition}
\begin{proof}
Follows immediately from the expression of $Z_{\infty}$ and Lemmas $1$ and $2$ of the reference \cite{andre1}.
\end{proof}
The region $\Lambda_{\infty}$ can be partitioned into arcs that goes from $p_j$ to the adjacent ones ($p_{j+1}$ and $p_{j-1}$).  So,
we will separate  $\Lambda_{\infty}$ into arcs, as follows:
	case: let $I_{2j} = \{(x, P_{\infty}(x)) |\;   j < x < j + 1\}$ and $I_{2j+1} = \{(x, -P_{\infty}(x)) |\; j < x < j +1\}$.

Before we continue, consider the following notation:
put $\mathcal{G}^{Z_{\infty}}$ as the graph whose vertices are the arcs of trajectories of ${\Omega}_{\infty}$ and the edges are the transitions from one arc to the other with a certain probability. Consider also the topological Markov shift $\mathcal{A}_{\mathcal{G}^{Z_{\infty}}}^{\mathbb{N}}$ associated with the graph $\mathcal{G}^{Z_{ \infty}}$ and $\Theta_{\infty} = \left\lbrace (x_j )_{j\in\N} ; x_j \in \mathcal{A}_{\mathcal{G}^{Z_{\infty}}}^{\N}\; \mbox{and}\: \mid x_{j+1} - x_j\mid \leq 2, \;\; \forall j \in \N\right\rbrace .$
	
\begin{definition} 
	\label{def7}
	Let $s : \Omega_{\infty} \rightarrow \Theta_{\infty}$ be given by $s(\gamma) = (s_j(\gamma))_{j \in \mathbb{N} }$, where : 
		
$$ s_j(\gamma) =  \left\{ \begin{array}{cc} \left.\begin{array}{l} 
	n \; \textrm{if}  \; \gamma(j) \; \in \; I_n  \\
	m \; \textrm{if}  \; \gamma(j) \; \in \; \lbrace (p_l, 0)\rbrace_{l \in \mathbb{N}} \; \textrm{and} \; \gamma \left( j + \frac{1}{2}\right) \in  I_m
				
			\end{array} \right.
		\end{array}.
		\right.   $$ 
		
	The sequence $s(\gamma)$ is called the itinerary of $\gamma$.
	\end{definition}

Now we must analyze whether, $s(\gamma) \in \Theta_{\infty}$. From Definitions \ref{def2} and \ref{def1}, of local and global trajectories, we see that if $\gamma$ goes through some compartment $I_{2l}$, then the next one must be the one below it $(I_{2l+1})$ or the one to its right $(I_{2l+2})$ and if it goes through $I_{2l+1}$ then the following one must be above it $(I_{2l})$ or the one to its left $(I_{2l-1})$. In any case, $ 0 < \mid s_j (\gamma) - s_{j+1}(\gamma)\mid \leq 2$, so $s(\gamma) \in \Theta_{\infty}$. From Proposition \ref{pro}, $(s_j (\gamma))_{j \in \mathbb{N}}$ encodes every compartment $I_j$ that the trajectory $\gamma$ visits in positive and negative time.	
	
According to Definition \ref{def7}, given $\gamma \in \Omega_{\infty}$, there exist infinitely many distinct trajectories with the same itinerary of $\gamma$, simply because the initial conditions belong to the same arc $I_n$. In order to avoid this problem we will consider the following definition:
	
	\begin{definition}
		\label{def8}
		Let $\gamma_1 , \gamma_2 \in \Omega_{\infty} .$ We say that $\gamma_1 \sim \gamma_2$ if and only if $s(\gamma_1) = s(\gamma_2) .$ Denote $\overline{\Omega}_{\infty} = \Omega_{\infty} / \sim .$
	\end{definition}  
	
The relation in Definition \ref{def8} is an equivalence relation. In fact, for  $\gamma_1 \in \Omega_k$, we have $\gamma_1 \sim \gamma_1$ because $s(\gamma_1) = s(\gamma_1)$. Also for all $\gamma_1, \gamma_2 \in \Omega_k$, if $\gamma_1 \sim \gamma_2$, we have $s(\gamma_1) = s(\gamma_2)$, and in this way, $\gamma_2 \sim \gamma_1$. And finally for all  $\gamma_1, \gamma_2, \gamma_3 \in \Omega_k$, if $\gamma_1 \sim \gamma_2$, we have $s(\gamma_1) = s(\gamma_2)$ and if $\gamma_2 \sim \gamma_3$, we have $s(\gamma_2) = s(\gamma_3)$, and in this way, $ s(\gamma_1) = s(\gamma_2) = s(\gamma_3)$, and therefore $\gamma_1 \sim \gamma_3 .$
	
\begin{proposition}\label{obs representante}
Note that, given $\overline{\gamma} \in \overline{\Omega}_{\infty}$, there exists a representative $\gamma^{\ast}$ such that  $\gamma^{\ast} (0) \in  \{(p_j,0)\}_{j \in \mathbb{N}}$.
\end{proposition}
\begin{proof}
Observe that, given $\gamma \in \Omega_{\infty}$, there exists a representative $\gamma^{\ast} \in \overline{\gamma}$ such that $\gamma^{\ast}(0) \in  \{(p_j,0)\}_{j \in \mathbb{N}}$, because if $\gamma(0) \in  \{(p_j,0)\}_{j \in \mathbb{N}} $, simply take $\gamma^{\ast}= \gamma$, otherwise, by Proposition \ref{pro}  there is an unique $t^{\ast}$  such that $\gamma(t^{\ast}) \in  \{(p_j,0)\}_{j \in \mathbb{N}}$. Moreover, if ${s}(\gamma) = ({s}_j)_{j \in \mathbb{N} }$, then $\gamma((t^{\ast} + j, t^{\ast} +j + 1)) = I_{{s}_j}$. Which implies that $\gamma((j, j + 1)) = I_{{s}_j}$ and, consequently, $\gamma^{\ast}(j+\frac{ 1}{2}) \in I_{{s}_j}$. So, ${s}(\gamma^{\ast}) = {s}(\gamma)$.
\end{proof}

	\begin{definition}
		\label{def10}
Define $\rho_{\infty} : \overline{\Omega}_{\infty} \times \overline{\Omega}_{\infty} \rightarrow \mathbb{R}$, by $$\rho_{\infty} (\overline{\gamma_1}, \overline{\gamma_2}) =  \underset {i \in \mathbb{N}}{\sum} \frac{d_i(\overline{\gamma_1}, \overline{\gamma_2})}{2^{ i }},$$where $d_i(\overline{\gamma_1}, \overline{\gamma_2}) = d_H \left( \gamma_1^{\ast}([i, i+1]),  \gamma_2^{\ast}([i, i+1])\right), $ $d_H$ is the Hausdorff distance and $\gamma_1^{\ast} , \gamma_2^{\ast}$ are those representatives given in Remark \ref{obs representante}.
	\end{definition}

For simplicity of notation, again we will refer only to $\gamma \in \overline{\Omega}_{\infty}$, meaning the equivalence class $\overline{\gamma}$ with the representative $\gamma^{\ast}$ given in Remark \ref{obs representante}. 

By Proposition $7$ \cite{andre1} , $(\overline{\Omega}_{\infty}, \rho_{\infty}) $ is a metric space.
Let $\overline{T_1} : \overline{\Omega}_{\infty} \rightarrow \overline{\Omega}_{\infty}$ be the function induced by $T_1$, that is, $\overline{T_1}(\overline{\gamma}) = \overline{T_1(\gamma)} .$ Note that the induced function does not depends on the representative. In fact if $s(\gamma_1) = s(\gamma_2) = (s_j)_{j \in \mathbb{N}}$, then, for all $j \in \mathbb{N}$ it happens
\[\gamma_1 (j), \gamma_2(j) \in I_{s_j} \Rightarrow \gamma_1 (j+1), \gamma_2 (j+1) \in I_{s_{j+1}} \Rightarrow\]\[  T_1 (\gamma_1)(j), T_2 (\gamma_2)(j) \in I_{s_{j+1}} \Rightarrow s(T_1(\gamma_1)) = s(T_1(\gamma_2)).\] The map $\overline{T}_1$
above is continuous (see Proposition $8$ \cite{andre1}).
	
Now, let $\overline{s} : \overline{\Omega}_{\infty} \rightarrow \Theta_{\infty}$ be the function induced by $s$, that is, $\overline{s}(\overline{\gamma}) =\overline{ s (\gamma)}$. Note that the induced function does not depend on the representative, because of the equivalence relation and because it is one-to-one. By Proposition $9$ \cite{andre1}, $\overline{s}$ in a homeomorphism over its image. 
										
	\section{Main results}
	\label{secao resultados principais}		
\begin{definition}
We define the Gurevich entropy $h_G$ of $Z=(X^{+},X^{-})$ in $M$, as the Gurevich entropy of $\overline{T}_1$ in $\overline{\Omega}$ such that $$h_G (Z):=h_G(\overline{\Omega}, \overline{T}_1).$$ Similarly, the Gurevich Pressure of $Z=(X^{+},X^{-})$ is defined by $$P_G(Z):= P_G ( \psi, \overline{T}_1, [\gamma_j]), $$ where $\gamma_j$ is an arc of trajectory of $Z=(X^{+},X^{-}).$
\end{definition}						
\begin{theorem}
\label{teo1}
Consider PSVF \eqref{zi}. Let $\psi: \widetilde{\Omega} \subseteq \overline{\Omega}_{\infty} \to \mathbb{R}$ be admissible and have summable variations. Consider $\mathcal{H}^{Z_{\infty}} := \mathcal{G}^{Z_{\infty}}|_{\widetilde{\Omega}}$ the subgraph of $\mathcal{G}^{Z_{\infty}} $, whose trajectory arcs belong to $\widetilde{\Omega}$ and suppose that the coutable Markov shift associated with $\mathcal{H}^{Z_{\infty}}$ is topologically mixing. Then, $$P_G\left(Z_{\infty}|_{\widetilde{\Omega}}\right)<\infty.$$
\end{theorem}

\begin{corollary}
\label{cor1}
With the hypotheses of the previous theorem, we obtain that there exists $\widetilde{\Omega} \subseteq \overline{\Omega}_{\infty}$, such that $$h_G\left(Z_{\infty}|_{\widetilde {\Omega}}\right)<\infty.$$
\end{corollary}	

\begin{theorem}
	\label{teo2}
Suppose that $\mathcal{A}_{\mathcal{G}^{Z_{\infty}}}^{\mathbb{N}}$ is topologically mixing, $\mathcal{G}^{Z_{ \infty} }$ is positively recurrent and $\mu_{\pi}$ is ergodic. Then the system $ ( \overline{\Omega}_\infty,\mathcal{D}, med, \overline{T}_1)$ is ergodically equivalent to the system $(\Theta_{\infty}, \mathcal{B}, \mu_{ \pi},\sigma_{\mathcal{G}^{Z_{\infty}}})$, that is, $\overline{T}_1: \overline{\Omega}_{\infty} \to \overline{\Omega}_{\infty } $ is a Bernoulli map. Furthermore, if the Markov shift $\mathcal{A}_{\mathcal{G}^{Z_{\infty}}}^{\mathbb{N}}$ is aperiodic and irreducible, $\overline{T} _1$ is strongly mixing and, therefore, weakly mixing and ergodic (see Definition \ref{weakly-strong-mix}).
\end{theorem}			

\begin{corollary}
	\label{cor2}
	With the hypotheses of the previous theorem, we get that $\overline{T}_1 : \overline{\Omega}_{\infty} \to \overline{\Omega}_{\infty}$ is topologically mixing.
\end{corollary}	

\begin{corollary}[Variational Principle]
\label{cor-teo2}
Consider PSVF \eqref{zi}. Let $\psi: \overline{\Omega}_{\infty} \to \mathbb{R}$ be admissible and have summable variations. Consider 
the graph  $\mathcal{G}^{Z_{\infty}}$, whose trajectory arcs belong to $\overline{\Omega}_{\infty}$ and suppose that the coutable Markov shift associated with $\mathcal{G}^{Z_{\infty}}$ is topologically mixing. Then, $$P_G\left(Z_{\infty}\right) = \sup \left\lbrace h_{med}\left( \sigma_{\mathcal{G}^{Z^{\infty}}}\right) + \int \psi \,d(med)  \right\rbrace ,$$ where the supremum ranges over all shift invariant Borel probability measures $med$ such that $\psi$ is $med$-integrable, and $\left( h_{med}\left( \sigma_{\mathcal{G}^{Z^{\infty}}}\right), \int \, \psi \,d(med)\right) \neq (\infty,-\infty).$
\end{corollary}

\begin{theorem}
\label{teoo3}
Let $\mathcal{A}_{\mathcal{G}^{Z_{\infty}}}^{\mathbb{N }}$ be the topological Markov shift associated with the graph $\mathcal{G}^{Z_{ \infty}} $. If there exists a Lyapunov function $\mathcal{V}$, then we can find $0 < \theta < 1$ such that for any $\eta > \theta$ and any trajectory arcs $\gamma_i, \gamma_j$
$$\parallel Prob(\mathcal{ X}_m = \gamma_j| \mathcal{ X}_0 = \gamma_i)-\pi(\gamma_i)\parallel_{TV} \leq \mathcal{V}(\gamma_i) \dfrac{\eta^{m+1}}{\eta - \theta}, $$ where $\pi$ is the stationary distribution. In particular, the transient distribution converges to a stationary distribution exponentially fast and the topologiocal Markov shift is mixing.
\end{theorem}
								
The previous results where established for very specific PSVFs. However, they can be extended to a bigger class of PSVFs following the next two results:

\begin{theorem}
\label{teo-princ}
Let $\widehat{Z}$ be a PSVF defined over a compact $2$-dimensional surface $M$ and let $\overline{\Upsilon}_{\infty}=\Upsilon_{\infty}/\sim$, where $ \Upsilon_{\infty} = \{\gamma : \mbox{positive global trajectory of}\; \widehat{Z}\; with\; \gamma(0) \in \overset{\infty}{\underset{i = 0}\bigcup} I_i, \;\mbox{where}\; I_i \;\; \mbox{is }$\\ $\mbox{an arc of}\;\mbox{trajectory} \} $ and $\sim$ is the equivalence relation given in Definition \ref{def8}.
If  $\widehat{Z}$  presents a $\infty$-homoclinic loop, the system $(\overline{\Upsilon}_{\infty},\mathcal{O}, med, \overline{\widehat{T}}_1)$ is ergodically equivalent to the system $(\mathcal{A}_{\mathcal{G}^{Z_{\infty}}}^{\mathbb{N}}, \mathcal{B}, \mu_{ \pi}, \sigma_{\mathcal{G}^{Z_{\infty}}})$ where  $\overline{\widehat{T}}_1$ is the time-one map associated with $\widehat{Z}$ and $\mathcal{O}$ is the $\sigma_{\mathcal{G}^{Z_{\infty}}}$-Borel algebra in $\overline{\Upsilon}_{\infty}.$ Furthermore, $\overline{\widehat{T}}_1$ is strong-mixing and therefore weakly-mixing and ergodic.

\begin{corollary}
\label{cor-teo5}
With the hypotheses of the previous theorem, we get that $\overline{\widehat{T}}_1 : \overline{\Upsilon}_{\infty} \to \overline{\Upsilon}_{\infty}$ is topologically mixing. Furthermore, 
\begin{enumerate}
\item[$(i)$] If $\widehat{\psi} : \widehat{\Omega}\subseteq \overline{\Upsilon}_{\infty} \to \mathbb{R}$ is admissible and have summable variations, then $P_G\left(Z|_{\widehat { \Omega}}\right)<\infty $ and the same happens with the Gurevich entropy, $h_G\left(Z|_{\widehat{\Omega}}\right)$;
\item [$(ii)$] $P_G\left(\widehat{Z}\right) = \sup \left\lbrace h_{med}\left( \sigma_{\mathcal{G}^{\widehat{Z}}}\right) + \int \widehat{\psi} \,d(med)  \right\rbrace ,$ where the supremum ranges over all shift invariant Borel probability measures $med$ such that $\widehat{\psi}$ is $med$-integrable, and $\left( h_{med}\left( \sigma_{\mathcal{G}^{\widehat{Z}}}\right), \int\, \widehat{\psi} \,d(med)\right) \neq (\infty,-\infty).$
\item [$(iii)$] Consider $\mathcal{A}_{\mathcal{G}^{\widehat{Z}}}^{\mathbb{N}}$ the topological Markov shift gas associated with the graph $\mathcal{G}^{\widehat{Z}} $ with stationary distribution $\pi$. If there exists a Lyapunov function $\mathcal{\widehat{V}}$ then we can find $0 < \theta < 1$ such that for any $\eta > \theta$ and any trajectory arcs $\widehat{\gamma}_i, \widehat{\gamma}_j$
$$\parallel Prob(\mathcal{X}_m = \widehat{\gamma}_j| \mathcal{ X}_0 = \widehat{\gamma}_i)-\pi(\widehat{\gamma}_i)\parallel_{TV} \leq \mathcal{\widehat{V}}(\widehat{\gamma}_i) \dfrac{\eta^{m+1}}{\eta - \theta}. $$
\end{enumerate}	

\end{corollary}

\end{theorem}

\section{Proof of the main results}\label{secao prova resultados principais}	

\subsection{ Proof of Theorem \ref{teo1} }
\begin{proof}
First, by Theorem $4.3$ \cite{sarig}, $P_G\left(Z_{\infty}|_{\widetilde{\Omega}}\right)$ exists for every trajectory arc $\gamma_1$, and is independent of the choice of $\gamma_j$. Now, let $\gamma_1, \gamma_2 \in \widetilde{\Omega}$, since $\psi$ is admissible and the topological Markov shift is topologically mixing, $$\mathcal{L}_{\psi}{\mathds{1}}(\gamma_1) =\underset{\sigma_{\mathcal{G}}^n\left(\gamma_2\right) =\gamma_1}{\sum} e^{\psi\left({\gamma_2}\right)}{\Large \mathds{1}}_{[\gamma_j]}(\gamma_2) < \infty, \quad \mbox{where}\quad  \gamma_j \quad\mbox{is any arc of trajectory}.$$ Thus, $$\left|\mathcal{L}_{\psi}{\mathds{1}}(\gamma_1)\right| = \left|\underset{\sigma_{\mathcal{G}}^n\left(\gamma_2\right) =\gamma_1}{\sum} e^{\psi\left({\gamma_2}\right)}{\Large \mathds{1}}_{[\gamma_j]}(\gamma_2)\right| < \infty.$$ Therefore, there exists $K>0$ such that $\left|\mathcal{L}_{\psi}\mathds{1}(\gamma)\right|< K.$ This way, $$\parallel\mathcal{L}_{\psi}\mathds{1}\parallel_{\infty} < \infty.$$ Therefore, again by Theorem $4.3$ \cite{sarig},  $P_G\left(Z_{\infty}|_{\widetilde{\Omega}}\right) < \infty.$
\end{proof}

\subsection{ Proof of Corollary \ref{cor1} }
\begin{proof}
The proof follows directly from the previous Theorem \ref{teo1} and the fact that the Gurevich entropy is equal to the Gurevich pressure for the null identical potential, i.e., $\psi\equiv 0$.
	
\end{proof}
								
\subsection{ Proof of Theorem \ref{teo2} }
\begin{proof}
Since $\mathcal{A}_{\mathcal{G}^{Z_{\infty}}}^{\mathbb{N}}$ is topologically mixing, it follows that $\mathcal{A}_{\mathcal{G }^{Z_{\infty}}}^{\mathbb{Z}}$ is topologically transitive and therefore $\mathcal{G}^{Z_{\infty}}$ is strongly connected. Furthermore, by the Corollary \ref{cor1}, $0< h_G(Z_{\infty})<\infty$, since $\mathcal{A}_{\mathcal{G}^{Z_{\infty}}} ^{\mathbb{N}}$ is topologically mixing and $\psi\equiv 0$ is admissible and has summable variations. Therefore, by Theorem $7$ \cite{gurevi} $\mathcal{A}_{\mathcal{G}^{Z_{\infty}}}^{\mathbb{N}}$ admits a single maximum invariant probability measure entropy $\mu_{\pi}$.
	
Now let $C \subset \mathcal{A}_{\mathcal{G}^{Z_{\infty}}}^{\mathbb{N}}$ be measurable and invariant in $\sigma_{\mathcal{G}} ^{Z_{\infty}}$-Borel algebra $\mathcal{B}$. Since $\mu_{\pi}$ is ergodic we can assume without loss of generality that $\mu_{\pi}(C) = 1.$	Note that from Proposition $9$ \cite{andre1} the map $\overline{s}^{-1}: \overline{s}(\overline{\Upsilon}_{\infty}) \to \overline{\Upsilon}_{\infty} $ is a topological conjugacy, i.e., $\overline{T}_1 \circ \overline{{s}}^{-1}= \overline{{s}}^{-1}\circ \sigma_{\mathcal{G}^{Z_{\infty}}}$.

Taking $B= \overline{{s}}^{-1}(C) \subset \overline{\Upsilon}_{\infty }$, then we have that \[\overline{T}_1(B) = (\overline{{s}}^{-1} \circ \sigma_{\mathcal{G}^{Z_{\infty}}} \circ \overline{{s}})(B) =  \overline{{s}}^{-1} ( \sigma_{\mathcal{G}^{Z^{\infty}}}(\overline{{s}}(B))) \subset  \overline{{s}}^{-1}(C) = B.\]As every $\sigma_{\mathcal{G}^{Z^{\infty}}}$-algebra is also an algebra, it follows from Theorem $1.1$ \cite{papa} that there are  $\widehat{\Upsilon}_1, \widehat{\Upsilon}_2, \cdots \in \mathcal{D}$, with $ \widehat{\Upsilon}_i \subseteq  \widehat{\Upsilon}_{i+1}$ such that $\mathcal{D} = \underset {i=1}{\overset{\infty}{\bigcup}} \widehat{\Upsilon}_i$ and $med(\widehat{\Upsilon}_i ) < \infty$, for all $i$. Put $\widehat{\mu} = \overline{{s}}^{-1}_{\ast}\mu_{\pi}$ which is a measure in $\overline{\Upsilon}_k$. Since $\overline{{s}}^{-1} : \overline{{s}}(\overline{\Upsilon}_{\infty}) \to \overline{\Upsilon}_{\infty} $ is a bijection, it follows that $\widehat{\mu} (\widehat{\Upsilon}_i ) < \infty$ , for all $i$. Therefore, by Theorem $2.4$ \cite{papa}, we get $med = \widehat{\mu}.$ In this way, it follows that \[med(B) = \overline{{s}}^{-1}_{\ast}\mu_{\pi}(B)= \mu_{\pi}((\overline{{s}}^{-1})^{-1}(B)) = \mu_{\pi}(\overline{{s}}(B)) = \mu_{\pi}(C) = 1.\]Therefore, $\overline{{s}}^{-1}$ is a bijection, restricted to a subset of total measure, and both it and its inverse are measurable.
In this way, $\overline{T}_1$ is a Bernoulli map, that is, the systems $(\overline{\Upsilon}_{\infty},\mathcal{D}, med, \overline{T}_1)$ and $(\mathcal {A}_{\mathcal{G}^{Z_{\infty}}}^{\mathbb{N}}, \mathcal{B}, \mu_{\pi}, \sigma_{\mathcal{G}^{Z_{\infty}}})$ are equivalent.

In order to conclude the result, it is enough to show that $\overline{T}_1$ is strong-mixing, becuase a strong-mixing map is also weakly-mixing and all weakly-mixing is ergodic. But, we have that $\overline{T}_1$ is strong-mixing, since $\sigma_k^{+}$ is strong-mixing. Indeed, given any measurable sets $\Omega_A,\Omega_B  \in \mathcal{D}$ we get
$$med\left( (\overline{T}_1)^{-n}(\Omega_A) \cap \Omega_B\right)= \mu_{\pi}\left((\overline{{s}}^{-1})^{-1}\left((\overline{T}_1)^{-n}(\Omega_A) \cap \Omega_B \right) \right)=  $$
$$\mu_{\pi}\left( (\sigma_{\mathcal{G}^{\infty}})^{-n}(\overline{{s}}(\Omega_A))\cap\overline{{s}}(\Omega_B)\right) \to \mu_{\pi}\left(\overline{{s}}(\Omega_A)\right)\mu_{\pi}\left(\overline{{s}}(\Omega_B)\right) = med(\Omega_A) med(\Omega_B),$$
when $n \to \infty.$
\end{proof}

\begin{remark}
If $\mu_{\pi}$ is not ergodic it will still be an invariant probability measure and from Theorem \ref{teo2} we can conclude that $med$ will also be an invariant probability measure and more than that unique.
\end{remark}	

\subsection{ Proof of Corollary \ref{cor2} }
\begin{proof}
The proof follows directly from the previous Theorem \ref{teo2}, since  $\mathcal {A}_{\mathcal{G}^{Z_{\infty}}}^{\mathbb{N}}$ is topologically mixing iff $\sigma_{\mathcal{G}^{Z_{\infty}}}$ is topologically mixing.
\end{proof}

\subsection{ Proof of Corollary \ref{cor-teo2} }
\begin{proof}
The proof follows directly from Theorem \ref{teo2}, together with Theorem $5.3$ \cite{sarig}.
\end{proof}
\subsection{ Proof of Theorem \ref{teoo3} }
\begin{proof}
Since there is a Lyapunov function, it follows from \cite{meyn} that there is a stationary distribution $\pi$. Therefore, according to Theorems $2.1$ and $2.2$ of \cite{mt}, the result follows.
\end{proof}
\subsection{ Proof of Theorem \ref{teo-princ} }
\begin{proof}
By item $(iii)$  of Theorem $B$ \cite{andre1}, it follows that $Z_{\infty}$ is $\Sigma$-equivalent to $Z$. Furthermore, by Theorem \ref{teo2}  we have that $(\overline{\Theta}_{\infty},\mathcal{D}, med, \overline{{T}}_1)$ is ergodically equivalent to $( \mathcal{A}_{\mathcal{G}^{Z_{\infty}}}^{\mathbb{N}}, \mathcal{B}, \mu_{\pi}, \sigma_{\mathcal{G}^{Z_{\infty}}})$. Therefore, since $Z_{\infty}$ is $\Sigma$-equivalent to $Z$, follows that $(\overline{\Upsilon}_{\infty},\mathcal{O}, med, \overline{\widehat{T}}_1)$ is ergodically equivalent to $( \mathcal{A}_{\mathcal{G}^{Z_{\infty}}}^{\mathbb{N}}, \mathcal{B}, \mu_{\pi}, \sigma_{\mathcal{G}^{Z_{\infty}}})$.
\end{proof}
\subsection{ Proof of Corollary \ref{cor-teo5} }
\begin{proof}
Given the hypotheses of Theorem \ref{teo-princ}, we simply reproduce the proofs of the Corollary \ref{cor2} of Theorem \ref{teo2}, to conclude that $\overline{\widehat{T}}_1 : \overline{ \Upsilon}_{\infty} \to \overline{\Upsilon}_{\infty}$ is topologically mixing.

Now as the Theorem \ref{teo-princ} holds and by hypothesis $\widehat{\psi} : \widehat{\Omega}\subseteq \overline{\Upsilon}_{\infty} \to \mathbb{R}$ is admissible with summable variations, the item $(i)$ follows directly from Theorem \ref{teo1} and its Corollary \ref{cor1}.
The item $(ii)$ follows directly from Corollary \ref{cor-teo2}.
Item $(iii)$ follows directly from Theorem \ref{teoo3}.
\end{proof}

The Theorem $2.8$ \cite{sarig} allows to reduce problems on two-sided shifts to problems for one-sided shifts. Therefore, we present the following example. 

\begin{example}
	\label{exemp0}
Let $\mathcal{G}^{Z_{\infty}}$ be the graph whose vertices are the trajectory arcs of the PSVF $Z_{\infty}$. Consider $\mathcal{A}_{\mathcal{G}^{Z_{\infty}}}^{\mathbb{N}}$ the topological Markov shift associated with $\mathcal{G}^{Z_{\infty}}$, with the time-one  map $\overline{T}_1$. For the potential $\psi :  \mathcal{A}_{\mathcal{G}^{Z_{\infty}}}^{\mathbb{N}} \to \mathbb{R}$ given by $\psi\left( I\right) = \log \left(I_0(I_0 + 1)\right), \, \mbox{where}, \, I=\left(I_0, I_1, I_2, \cdots\right) $, it follows from Example $1$ \cite{barreira} that 
$$ P_G(-\psi) = \log \left( \overset{\infty}{\underset{n = 1}{\sum}} \dfrac{1}{n(n+1)}\right) = \log(1) = 0.$$

\end{example}

\section{Hausdorff dimension for $Z_{\infty}$}
\label{hausdorf-dim}
Strongly positive recurrent topological Markov shifts have the property that the entropy is concentrated within the system and not close to infinity. In fact, given a graph $\mathcal{G}$ associated with topological Markov shifts. Gurevich and Zargaryan \cite{gurevizar} showed that an equivalent condition for strongly positive recurrent topological Markov shifts is the existence of a connected finite subgraph $\mathcal{H} \subset \mathcal{G}$ such that there are more paths inside than outside of $\mathcal{H}$ (in terms of exponential growth). On the other hand, for graphs that are not strongly positive recurrent entropy is supported by the infinite paths that spend most of their time outside a finite subgraph (see \cite{ruet} , Proposition $3.3$).
Following the lines of observations \cite{iommi} (see Proposition 2.20) the authors characterize strongly positive recurrent topological Markov shifts as those that have topological entropy at infinity (see Definition \ref{def-entrop-topo-inf}) strictly smaller than topological entropy.	
In this section we will relate the topological entropy at infinity of the global trajectories of $Z_{\infty}$ with the Hausdorff dimension of the set of recurrent points that escape on average.

\begin{definition}
\label{def-entrop-topo-inf}
Let $\left( \mathcal{A}_{\mathcal{G}^{Z_{\infty}}}^{\mathbb{N}}, \sigma_{\mathcal{G}^{Z^{\infty}}}\right) $ be a topological Markov shift. Let $\mathcal{M}, q \in \mathbb{N}$. For $n \in \mathbb{N}$ let $z_n(\mathcal{M}, q)$ be the number of cylinders of the form $[a_0, \cdots, a_{n+1}]$, where $a_0\leq q, a_{n+1} \leq q,$ and $$\#\left\lbrace i \in \left\lbrace 0,1, \cdots, n+1\right\rbrace : a_i \leq q  \right\rbrace \leq \dfrac{n+2}{\mathcal{M}}. $$
Define
$$h_{\infty}(\mathcal{M}, q) := \underset{n \to \infty}{\limsup} \dfrac{1}{n}\log z_n \left(\mathcal{M},q \right) $$ and $$h_{\infty}(q) := \underset{\mathcal{M} \to \infty}{\liminf} \,h_{\infty} \left(\mathcal{M}, q \right). $$
The \textbf{topological entropy at infinity} of $\left( \mathcal{A}_{\mathcal{G}^{Z_{\infty}}}^{\mathbb{N}}, \sigma_{\mathcal{G}^{Z^{\infty}}}\right) $ is defined by $h_{\infty} := \underset{q \to \infty}{\liminf}\,h_{\infty}(q).$
\end{definition}

\begin{definition}
Let $\left( \mathcal{A}_{\mathcal{G}^{Z_{\infty}}}^{\mathbb{N}}, \sigma_{\mathcal{G}^{Z^{\infty}}}\right) $ be a topological Markov shift, the set of points that Let $\left( \mathcal{A}_{\mathcal{G}^{Z_{\infty}}}^{\mathbb{N}}, \sigma_{\mathcal{G}^{Z^{\infty}}}\right) $ be a topological Markov shift \textbf{escape on average} is defined by $$E:= \left\lbrace \gamma \in \mathcal{A}_{\mathcal{G}^{Z_{\infty}}}^{\mathbb{N}}: \underset{n \to \infty}{\lim} \overset{n-1}{\underset{i = 0 }{\sum}} \mathds{1}_{[\gamma_j]}  = \left( \sigma^{i}(\gamma)\right) = 0\;\; \mbox{for every}\;\; \gamma_j \in  \mathcal{A}_{\mathcal{G}^{Z_{\infty}}} 
 \right\rbrace. $$
We say that $\gamma \in \mathcal{A}_{\mathcal{G}^{Z_{\infty}}}^{\mathbb{N}}$ is a \textbf{recurrent} point if there exists an increasing sequence $(n_k)_{k \in \mathbb{N}}$ such that $ \underset{n \to \infty}{\lim}\sigma^{n_k}(\gamma) = \gamma.$ The set of recurrent points is denoted by $\mathcal{R}$.
\end{definition}

\begin{theorem}
	\label{teo-dim-haus-1}
Let $\left( \mathcal{A}_{\mathcal{G}^{Z_{\infty}}}^{\mathbb{N}}, \sigma_{\mathcal{G}^{Z_{\infty} }}\right) $ be a topological Markov shift, topologically transitive with finite Gurevich entropy, then $$dim_{H} \left( \mathcal{R} \right) \leq \dfrac{h_{G}\left( \mathcal{A}_{\mathcal{G}^{Z_{\infty}}}^{\mathbb{N}},\sigma_{\mathcal{G}^{Z_{\infty}}}\right)}{\log(2)}. $$
\end{theorem}
\begin{proof}
Since $Z_{\infty}$ is invariant by topological conjugation, the proof follows directly from Theorem $9$ \cite{andre1} and Theorem $3.1$ \cite{iom}.
\end{proof}
 
\begin{theorem}
\label{teo-dim-haus-2}
Let $\left( \mathcal{A}_{\mathcal{G}^{Z_{\infty}}}^{\mathbb{N}}, \sigma_{\mathcal{G}^{Z_{\infty}}}\right) $ be a topological Markov shift, topologically transitive such that $h_{\infty} < \infty$. Then $$dim_{H} \left( E \cap \mathcal{R} \right) \leq \dfrac{h_{\infty}}{\log(2)}, $$ where $dim_H$ denotes the Hausdorff dimension with respect to the metric \eqref{metrica}.
\end{theorem}

\begin{proof}
Since $Z_{\infty}$ is invariant by topological conjugation, the proof follows directly from Proposition $9$ \cite{andre1} and Theorem $8.9$ \cite{iommi}.
\end{proof}

\begin{corollary}
	\label{cor-haus}
Let $\left( \mathcal{A}_{\mathcal{G}^{Z_{\infty}}}^{\mathbb{N}}, \sigma_{\mathcal{G}^{Z_{\infty}}}\right) $ be a topological Markov shift, topologically transitive strongly positive recurrent with finite Gurevich entropy, then $$dim_{H} \left( E \cap \mathcal{R} \right)  < dim_{H} \left(  \mathcal{R} \right).$$
\end{corollary}
\begin{proof}
Since, $\left( \mathcal{A}_{\mathcal{G}^{Z_{\infty}}}^{\mathbb{N}}, \sigma_{\mathcal{G}^{Z_{\infty}}}\right) $ be a topological Markov shift, topologically transitive strongly positive recurrent with finite Gurevich entropy, It follows from Proposition $20$ \cite{iommi} that $$h_{\infty} < h_{G}\left( \mathcal{A}_{\mathcal{G}^{Z_{\infty}}}^{\mathbb{N}},\sigma_{\mathcal{G}^{Z_{\infty}}}\right) < \infty \Rightarrow \dfrac{h_{\infty}}{\log (2)} < \dfrac{ h_{G}\left( \mathcal{A}_{\mathcal{G}^{Z_{\infty}}}^{\mathbb{N}},\sigma_{\mathcal{G}^{Z_{\infty}}}\right)}{\log (2)} < \infty .$$  Therefore, by Theorem \ref{teo-dim-haus-1} and Theorem \ref{teo-dim-haus-2}, we have $$dim_{H} \left( E \cap \mathcal{R} \right) \leq dim_{H} \left( \mathcal{R} \right).$$
\end{proof}

\begin{remark}
In Theorem \ref{teo-dim-haus-1}, Theorem \ref{teo-dim-haus-2} and the Corollary \ref{cor-haus} we use the PSVF $Z_{\infty}$, but the Theorem \ref{teo-princ} guarantees that the result also holds for any PSVF $Z$ that presents a $\infty$-homoclinic loop.
\end{remark}


\section{Gurevich entropy of PVFS with cofinite topology.}
\label{cofinity-topology}
So far in our approach to the symbolic dynamics of infinite alphabets, we have adopted the discrete topology, which is always meterizable, which allowed us to carry out an analysis considering the \eqref{metrica} metric. As we know if a set $A$ is finite, then the cofinite topology is the discrete topology. However, if $A$ is infinite, then the cofinite topology is not metrizizable. Therefore, an approach from the point of view of cofinite topology becomes interesting in the context of PSVFS.   For the case of topological Markov shifts induced by countably infinite graphs, \cite{reza} has made an approach that yields the same entropy as Gurevich's approach, providing formulas for the entropy of countable topological Markov shifts in terms of the spectral radius in $l^2$.

Below we show that in this context it is still possible to obtain a topological conjugation between the space of Markov shifts and the set of all global trajectories of the PSVF $Z_{\infty}$ and we can thus calculate the Gurevich entropy when we consider the cofinite topology.

Before proceeding, let us consider the following: Let $\mathcal{G} = G(V, E)$ be a finite or countable directed graph in which every vertex has at least one outgoing and one incoming edge.
\begin{definition}
The \textbf{two-sided topological Markov shift} associated to $\mathcal{G}$ is the
set $$\mathcal {A}_{\mathcal{G}^{Z_{\infty}}}^{\mathbb{Z}} := \left\lbrace x=(x_j)_{j \in \Z} \in V^{\Z};  x_i \to x_{i+1}\;\; \mbox{for all} i \in \Z\right\rbrace, $$ together with the metric $d(x, y) :=
2^{- \min\{\mid m \mid,\, x_m \neq y_m\}}$ and the action of the left-shift map $\sigma_{\mathcal{G}}(x_j) = b_j$, where $b_j = x_{j+1}$.
\end{definition}

The conditions for the compactness, topological transitivity, and topological
mixing of $\mathcal {A}_{\mathcal{G}^{Z_{\infty}}}^{\mathbb{Z}}$ are the same as for $\mathcal {A}_{\mathcal{G}^{Z_{\infty}}}^{\mathbb{N}}$. Local compactness is different: a two-sided topological Markov shift is locally compact iff for every a, $\#\{(u, v) : u \to a \to v\} < \infty$. Cylinders
are also slightly more complicated because of the need to keep track of the left-most
coordinate of the constraint. We will use the notation 
$$[a_0, \cdots , a_{n-1}]^{m} := \{ x \in \mathcal {A}_{\mathcal{G}^{Z_{\infty}}}^{\mathbb{Z}} : x_{m+i} = a_i, \;i = 0, \cdots, n - 1\}.$$
\begin{proposition}
The function $\overline{T}_1: \overline{\Omega}_{\infty} \to \overline{\Omega}_{\infty}$, where $\overline{\Omega}_{\infty}$ is equipped with the cofinite topology and is a homeomorphism.
\end{proposition}
\begin{proof}
As $\overline{\Omega}_{\infty}$, is endowed with the cofinite topology, it follows that $\overline{\Omega}_{\infty}$, for any $\gamma, \widetilde{\gamma} \in 
\overline{\Omega}_{\infty}$ distinct, there exists an open $U$ such that $ \gamma \in U$ and $\widetilde{\gamma} \notin U.$ But a consequence of this fact is that for every ${\gamma} \in \overline{\Omega}_{\infty} , \{{\gamma} \}$ is closed, so $\mathcal{U} = \overline{\Omega}_{\infty } \setminus \{{\gamma} \}$ is open. Let $\widehat{\gamma} \in {U}$, note that $\overline{T}_1^{-1}\left( \widetilde{\gamma}\right)(\cdot ) = \widetilde {\gamma}(\cdot -1) \in {U}$ which is open. Therefore $\overline{T}_1^{-1}({U})$ belongs to the cofinite topology and $\overline{T}_1$ is continuous. In an entirely analogous way, it is shown that its inverse $\overline{T}_1^{-1}$ is also continuous.
\end{proof}	
\begin{proposition}
The map $\overline{s}:\overline{\Omega}_{\infty} \to \overline{s}(\overline{\Omega}_{\infty}) $ is a homeomorphism onto its image.
\end{proposition}
\begin{proof}
Let $\left\lbrace \left(x_j\right)_{j \in \mathbb{Z}}\right\rbrace $ be closed in $\overline{s}\left(\overline{\Omega}_{\infty} \right)$, then $\sigma^{-1}\left(\left\lbrace\left(x_j\right)_{j \in \mathbb{Z}}\right\rbrace\right)$ is closed in $ \overline{s}\left(\overline{\Omega}_{\infty}) \right)$, since $\sigma$ is continuous. Note that $$\overline{s}^{-1} \left(\sigma^{-1}\left(  \left\lbrace \left(x_j\right)_{j \in \mathbb{Z}}\right\rbrace\right)\right)   = \overline{s}^{-1}\left(\{ \gamma\}\right).$$ Therefore, $\overline{s}$ is continuous. Similarly, it is shown that its inverse $\overline{s}^{-1}$ is continuous.
\end{proof}	
\begin{proposition}
\label{prop}
The function $\overline{{s}}: \overline{\Omega}_{\infty} \to \overline{{s}}\left( \overline{\Omega}_{\infty}\right) $ is a conjugacy between $\overline{T}_1$  and $\sigma$, i.e., $\overline{{s}} \circ \overline{T}_1 = \sigma \circ \overline{{s}}.$
\end{proposition}

\begin{proof}
Let $\gamma \in \overline{\Omega}_k, (a_j)_{j \in \mathbb{Z}} = \overline{{s}}(\gamma)$ and $(b_j)_{j \in \mathbb{Z}} = \overline{{s}}(\overline{T}_1(\gamma))$, then : 	
\begin{center}
\label{diagra}
$\xymatrix{
\overline{\Omega}_{\infty} \ar[d]_{\overline{{s}}} \ar[r]^{\overline{T}_1} & \overline{\Omega}_{\infty} \ar[d]^{\overline{{s}}} \\
\overline{{s}} (\overline{\Omega}_{\infty})  \ar[r]^{\sigma}& \overline{{s}} (\overline{\Omega}_{\infty})}$
	\end{center} 
$$b_j = r\; \textrm{if}\; \overline{T}_1(\gamma)(j) \in I_r \Rightarrow b_j = r \; \textrm{if}\; \gamma(j+1) \in I_r \Rightarrow	b_j = a_{j+1} \Rightarrow (b_j)_{j \in \mathbb{Z}} = \sigma\left( (a_j)_{j \in \mathbb{Z}}\right).$$
	
It remains to show that $\overline{{s}}\left( \overline{\Omega}_{\infty}\right)$ is a subshift. Since cofinite topology transforms each set into a separable compact space, by the first part of this proof, we have $\sigma = \overline{{s}} \circ \overline{T}_1\circ \overline{{s}}^{-1}$, which proves the invariant part. So, $\overline{{s}}$ is a conjugacy.
\end{proof}		
Now, we consider countably infinite graphs which we assume, without loss of generality, that their set of vertices is $\mathbb{N}$.

In the infinite-dimensional setting the choice of topology becomes important.
Let $l^2$ denote the Hilbert space of sequences $\{x_i\}_{i=1}^{\infty}$ such that $\sum_{i =1} \mid x_i \mid^2 < \infty$ equipped with the norm $\parallel x \parallel_{2} = \sqrt{\sum_{i =1} \mid x_i \mid^2}$. Then the adjacency operator $F : l^2 \to l^2$ is defined by	$$\left( Fx\right)_i  = \overset{\infty}{\underset{j = 1}{\sum}} F_{ij}x_j.$$
We suppose that $F$ is uniformly locally finite, i.e., there is a common upper bound for the number of successors of every vertex; therefore, the adjacency operator $F : l^2 \to l^2$ is bounded (\cite{mo} Theorem $3.2$).

The spectrum of a bounded operator $B : l^2 \to l^2$ is
$$Spec(B) = \left\lbrace \lambda \in \mathbb{C}; B-\lambda I\;\; \mbox{is not invertible}\right\rbrace, $$

and its spectral radius is the number $$\rho(B) = \underset{\lambda}{\sup}\left\lbrace \mid \lambda\mid \; \lambda \in Spec(B) \right\rbrace $$	
\begin{example}
Consider the subshift $\widetilde{\Theta} \subseteq \overline{{s}}\left( \overline{\Omega}_{\infty}\right) \subseteq \Theta_{\infty}$ defined by $$\widetilde{\Theta}:= \left\lbrace (x_j)_{j \in \mathbb{Z}} \in \overline{{s}}\left( \overline{\Omega}_{\infty}\right); \mid x_{j+1} - x_j\mid = 1 \;\; \forall j \in \mathbb{Z}\right\rbrace.$$
This subshift is generated by walks along an infinite graph, the infinite two-way path, whose spectral radius is $2$, (see \cite{reza}). Thus, there exists $\widetilde{\Omega} \subseteq \mathcal{A}_{\mathcal{G}^{Z_{\infty}}}^{\mathbb{Z}}$ for which we have $$h_G( \widetilde{\Omega}, \overline{T}_1) = h_G(\widetilde{\Theta}, \sigma_{\widetilde{\Theta}}) = \log (2).$$
\end{example}	

\section{Appendix}
\label{apen}

\subsection{Stocastic Process and Markov Chains}
\label{part}


A set of random variables $(\mathcal{X}_m)_{m \in M}$ indexed by a set $M$, defined in a probability space $(\Xi, \mathcal{F}, Prob )$ and taking values in a set $E$, called the state space, is a \textbf{stochastic process}. \textbf{Markov chains} are sequences of random variables  that possess the so-called Markov property: given one term in the chain (the present), the subsequent terms (the future) are conditionally independent of the previous terms (the past). In this work, we consider $ E= M = \mathbb{Z}$. Let us start with the formal definition. 

\begin{remark}
We can always associate a Markov chain with the shift space, where the random variable $\mathcal{X}_{m+1}$ of the sequence is equal to $\sigma\left( \mathcal{X}_m\right).$
\end{remark}

\begin{definition}
A stochastic process $(\mathcal{X}_m)_{m \in \mathbb{Z}}$  have the \textit{Markovian property} if
\begin{equation*}
Prob\left(\mathcal{X}_{m+1} = j \mid \mathcal{X}_0= i_0, \cdots, \mathcal{X}_{m-1} = i_{m-1}, \mathcal{X}_m = i \right) = Prob\left( \mathcal{X}_{m+1} = j \mid \mathcal{X}_m = i \right),
 \end{equation*}
for $m \in \mathbb{Z}$ and  $i, j, i_0, \cdots i_{m-1} \in \mathbb{Z}$.
\end{definition}

A sequence of random variables $( \mathcal{X}_m)_{m \in \mathbb{Z}}$ is a \textbf{Markov chain} if it has the Markovian property.   A Markov chain is defined by its stochastic matrix $W=(w(i,j)),$ where $w(i,j) =  Prob\left( \mathcal{X}_{m+1} = j \mid \mathcal{X}_m = i \right) \quad \forall \, i,j \in \mathbb{Z} \quad \textrm{and} \quad m \in \mathbb{Z},$  we can simply use matrix multiplication to get the desired result and $ w_{ij}^{(n)}$ will denote the element of row $i$ and column $j$ of $W^n, \, n$ positive integer.

\begin{definition}
We say that the states $i$ and $j$ communicate if $Prob(\mathcal{X}_m = i \mid \mathcal{X}_0=j)\quad \textrm{and} \quad Prob(\mathcal{X}_{\widetilde {m}} = j \mid \mathcal{X}_0=i)$ for some $m$ and $\widetilde{m}$. That is, from one of them we can eventually visit the other.
A Markov chain is said to be \textbf{irreducible} if all its states communicate.
\end{definition}

\begin{definition}
A state $i \in \mathbb{Z}$ is said to be \textbf{aperiodic} if it exists $n_0 \geq 1$ such that $w^{(n)}_{ii} >0$ for all $n \geq n_0$. A matrix $W$ is \textbf{aperiodic} if this property holds for all $i \in \mathbb{Z}$. Finally, a Markov chain is \textbf{aperiodic} if its stochastic matrix is aperiodic.
\end{definition}

\begin{definition}
	\label{tra}
The probability of first visit at state $j$ after $m$ steps, starting from state $i$,is:
$r_{ij}^{(m)}=Prob(\mathcal{X}_n=j, \mathcal{X}_1 \neq j, \mathcal{X}_2 \neq j, \cdots , \mathcal{X}_{m-1} \neq j  \mid \mathcal{X}_0=i).$
\end{definition}

The expected number of steps to arrive for the first time at state $j$ starting from $i$ is $h_{ij} = \underset{m > 0}{\sum} m \cdot r_{ij}^{(m)}.$
The probability of a visit (not necessarily for the first time) at state $j$, starting from state $i$, is: $f_{ii} = \underset{m > 0}{\sum} r_{ij}^{(m)}.$
If $f_{ij} < 1$ then there is a positive probability that the Markov chain never arrives at state $j$, so in this case $h_{ij}= \infty.$ A state $i$ for which  $f_{ii} < 1$ (i.e. the chain has positive probability of never visiting state $i$ again) is a \textbf{transient state }. If $f_{ii}=1$ then the state is  called \textbf{recurrent}. More so, if state $i$ is recurrent, but $h_{ii}= \infty$ is \textbf{recurring null}. If is recurrent and $h_{ii} \neq \infty$ is \textbf{positive recurrent}. Every state is either recurrent or transient. Assim, cada estado $j$ em uma cadeia de Markov contável é recorrente ou transitório, e é recorrente se e somente se um eventual retorno a $j$ (condicional em $\mathcal{X}_0=j$) ocorre com probabilidade $1.$

Now we will present a simple way to calculate the unique stationary distribution for a Markov chain on $\mathbb{Z}$. The concept of stationary distribution is almost identical to that introduced in the case of a finite state space. 

Let $(\mathcal{X}_m)_{m \in \mathbb{Z}}$ be a  irreducible aperiodic Markov chain with state space $\mathbb{Z}$ and stochastic matrix $W,$ and let $\pi_m$ be the distribution of $\mathcal{X}_m$: $\pi_m(i) = Prob(\mathcal{X}_m = i) \quad \textrm{for all} \;\; i \in \mathbb{Z}.$ If, for a given double sequence of transition probabilities $w(i,j)$, there is an initial distribution $\pi_j$ such that the distribution of all the terms of the chain is equal to the initial distribution, then $\pi_j$ is called a \textbf{stationary distribution} of the chain. Furthermore, we have that 
$$\pi_j = \displaystyle \sum _{i =1}^{\infty }\pi_{i}w(i,j). $$
\begin{proposition}
If a Markov chain with a countable state space is irreducible and positive recurrent, then it has a unique stationary distribution.
\end{proposition}
\begin{proof}
See \cite{gal}
\end{proof}

In the case of a finite state space, irreducibility is sufficient to obtain the uniqueness of the stationary distribution. In the countable case, we need to add the positive recurrence requirement.
\begin{proposition}
If a Markov chain with a countable state space is irreducible, positive recurrent and aperiodic, then, irrespective of the initial distribution $\pi_j.$ $$\underset{m \to \infty}{\lim} \pi_m(j) = \pi_j,$$
for any $j$, where $\pi_j$ is the unique stationary distribution of the chain.
\end{proposition}
\begin{proof}
	See \cite{gal}
\end{proof}
%


\begin{definition}
	
The \textbf{total variation distance} between two probability distributions $\pi$ and $\widetilde{\pi}$ on $\mathbb{Z}$ is defined as $
\| \pi - \widetilde{\pi} \|_{TV} = \underset{\mathcal{E}\subseteq \mathbb{Z}}{\max} \mid \pi(\mathcal{E}) - \widetilde{\pi}(\mathcal{E}) \mid.$
This definition is explicitly probabilistic: the distance between $\pi$ and $\widetilde{\pi}$ is the maximum difference between the probabilities assigned to a single event by the two distributions.
\end{definition}
	
\begin{definition}
	\label{trans}
A probability distribution $\pi$ is said to be \textbf{reversible} if, for all $i,j \in \mathbb{Z}$ gives $\pi(i)w(i,j) = \pi(j)w(j,i).$ When a Markov chain has a reversible probability distribution it is considered a reversible Markov chain.
\end{definition}

\begin{definition}
Let $\mathcal{X}$ be a  random variable, with values possible $x_1, x_2, \cdots $ and $w(x_i) = Prob(\mathcal{X} = x_i), i = 1, 2, \cdots , n , \cdots$. So, the expected value of $\mathcal{X}$, denoted by $\mathbb{E}(\mathcal{X})$ is defined as $\displaystyle \mathbb{E}[\mathcal{X}]=\sum _{i =1}^{\infty }x_{i}w(x_{i}).$
\end{definition}

Covariance is a statistical measure where you can compare two random variables, allowing you to understand how they relate to each other. 
Therefore, given an invariant probability $\mu$ of a dynamical system $g : M \to M$ and given measurable functions $ \mathfrak{j}, \mathfrak{h} : M \to \R$, we will  analyze the evolution of the covariance $C_n(\mathfrak{j}, \mathfrak{h}) = C(\mathfrak{j} \circ g^n, \mathfrak{h})$ when time $n$ goes to infinity. If $\mathfrak{j} = \mathds{1}_A$ and $\mathfrak{h} = \mathds{1}_B $ are indicator functions, then $\mathfrak{j}(x)$ gives information about the position of the initial point $x $, while $\mathfrak{h}(g^n(x))$ tells us the position of the $n$-th iterate $g^n(x)$.

\begin{definition}
\label{weakly-strong-mix}
Let $g : M \to M$ be a measurable map and let $\mu$ be an invariant probability. The sequence of covariances of two measurable functions $\mathfrak{j}, \mathfrak{h} : M \to \R$ is defined by 
$C(\mathfrak{j}, \mathfrak{h}) = \int \left(\mathfrak{j} \circ g^n \right)\mathfrak{h} \, d\mu - \int \mathfrak{j} \, d\mu  \int \mathfrak{h} \, d\mu \quad ; \quad n \in \mathbb{N}.$
We say that the system $(g, \mu)$ is
\begin{enumerate}
\item [$(i)$]weakly-mixing if $\underset{n \to \infty}{\lim}C_n(\chi_A, \chi_B) = \underset{n \to \infty}{\lim} \frac{1}{n}\overset{n-1}{\underset{i = 0 }{\sum}} \left| \mu \left( g^{-i}(A) \cap B\right)-\mu(A)\mu(B)\right| = 0; $
		\\	
\item [$(ii)$] strong-mixing if $\underset{n \to \infty}{\lim}C_n(\chi_A, \chi_B) = \underset{n \to \infty}{\lim} \mu \left( g^{-n}(A) \cap B\right)-\mu(A)\mu(B) = 0. $
\end{enumerate} 
for any measurable sets $A,B \subset M$. 
\end{definition}
		
									

\end{document}